\documentclass[11pt]{amsart} 
\usepackage{amsmath,amssymb,a4wide} 
\usepackage{mathabx}
\usepackage{color}
\usepackage{bbm}
\usepackage{xcolor,graphicx}
\usepackage{subfig}
\usepackage{mathrsfs}
\usepackage{tikz,pgfplots}
\usetikzlibrary{plotmarks}
\pgfplotsset{compat=newest}
\newtheorem{theorem}{Theorem} 
\newtheorem{lemma}[theorem]{Lemma} 
 
\newtheorem{remark}[theorem]{Remark}

\newtheorem{hypo}[theorem]{Hypotheses}
\usepackage{multirow}
\newsavebox{\measurebox}

\newcommand{\dt}{\delta t}

\newcommand{\Z}{{\mathbb Z}} 
\newcommand{\R}{{\mathbb R}} 
\newcommand{\C}{{\mathbb C}} 
\newcommand{\N}{{\mathbb N}} 
 
\newcommand{\T}{{\mathbb T}} 

\newcommand{\dd}{{\rm d}}

\usepackage{color}

\title[Energy preserving methods for NLSE]
{Energy preserving methods for nonlinear Schr\"odinger equations}

\author[C. Besse]{Christophe Besse}
\address[C. Besse]{Institut de Math\'ematiques de Toulouse, UMR5219, Universit\'e de Toulouse, CNRS UPS IMT, F-31062 Toulouse Cedex
9\\ France} 
\email{christophe.besse@math.univ-toulouse.fr}

\author[S. Descombes]{St\'ephane Descombes}
\address[S. Descombes]{Universit\'e C\^ote d'Azur, CNRS, INRIA, LJAD,  France}
\email{stephane.descombes@univ-cotedazur.fr}

\author[G. Dujardin]{Guillaume Dujardin}
\address[G. Dujardin]{
        Inria, Univ. Lille, CNRS, UMR 8524, Laboratoire Paul Painlev\'e, F-59000 Lille, France}
\email{guillaume.dujardin@inria.fr}

\author[I. Lacroix-Violet]{Ingrid Lacroix-Violet}
\address[I. Lacroix-Violet]{Laboratoire Paul Painlev\'e, CNRS UMR 8524, INRIA RAPSODI Team,
  	Universit\'e de Lille 1,
  	Cit\'e Scientifique, 59655 Villeneuve d'Ascq Cedex, France}
\email{Ingrid.Violet@math.univ-lille1.fr}
      
\begin{document}


\begin{abstract}
This paper is concerned with the numerical integration in time of nonlinear Schr\"odinger equations using different methods preserving the energy or a discrete analog of it.
The Crank-Nicolson method is a well known method of order $2$ but is fully implicit  and one may prefer a linearly implicit method like the relaxation method introduced in \cite{BessePhD} for the cubic nonlinear Schr\"odinger equation. This method is also an energy preserving method and numerical simulations have shown that its order is $2$. In this paper we give a rigorous proof of the order of this relaxation method and propose a generalized version that allows to deal with general power law nonlinearites. Numerical simulations for different physical models show the efficiency of these methods.
\end{abstract}


\maketitle
{\small\noindent 
{\bf AMS Classification: 35Q41, 81Q05, 65M70} 

\bigskip\noindent{\bf Keywords.} Nonlinear Schr\"odinger equation, Gross-Pitaevskii equation, numerical methods, relaxation methods.

\section{Introduction}
The nonlinear Schr\"odinger equation (NLSE) is a fairly general dispersive
partial differential equation arising in many areas of physics and
chemistry \cite{ADP,AS,DP,PitaevskiiStringari,SulemSulem}. 
One of the most important application of the NLSE is for laser beam
propagation in nonlinear and/or quantum optics and  there it is also
known as parabolic/paraxial approximation of the Helmholtz or
time-independent Maxwell equations \cite{ADP,AS,DP,SulemSulem}. 
In the context of the modeling of the Bose-Einstein condensation (BEC),
the nonlinear Schr\"odinger equation is known as the Gross-Pitaevskii
equation (GPE), which is a widespread model that 
describes the averaged dynamics of the condensate
\cite{PitaevskiiStringari,XavRom2015}. 
Depending on the physical situation that one considers,
several terms in the right-hand side of the NLSE appear.
Our goal is to develop numerical methods for the time integration
of a fairly general NLSE including realistic physical situations,
that have high-order in time and have good qualitative properties
(preservation of mass, energy, etc) over finite times.
We develop our analysis in spatial dimension $d\in\{1,2,3\}$
because it fits the physical framework, even if most methods and results
naturally extend to higher dimensions.
The space variable $x$ will sometimes lie in $\R^d$,
sometimes on the $d$-dimensional torus $\T_\delta^d=(\R/(\delta\Z))^d$
(for some $\delta>0$).
In this paper, we consider a NLSE of the form
\begin{equation}
\label{eq:GPE}
  i\partial_t \varphi (t,x) = \left(-\frac{1}{2}\Delta
+ V(x)
+ \beta |\varphi|^{2\sigma}(t,x)
+ \lambda \left(U \ast |\varphi(t,\cdot)|^2\right)(x)
-\Omega . R \right)
\varphi(t,x),
\end{equation}
where
$\varphi$ is an unknown function from $\R\times\R^d$ or $\R\times\T_\delta^d$
to $\C$, $\Delta$ is the Laplace operator, $V$ is some real-valued potential
function, $\beta\in\R$ is a parameter that measures the local nonlinearity
strength, $\lambda\in\R$ is a parameter that measures the nonlocal
nonlinearity strength with convolution kernel $U$,
$\Omega\in\R^d$ is a vector encoding the direction and the speed of a rotation,
and $R$ is a rotation operator that is important in the modeling
of rotating BEC (for example, $R= x \wedge (-i\nabla)$ when $d=3$).
The NLSE is supplemented with an initial datum $\varphi_{\text{in}}$.
The results presented in this paper extend to more general power law nonlinearities such as
$$
  i\partial_t \varphi (t,x) = \left(-\frac{1}{2}\Delta
+ V(x)
+ \sum_{k=1}^K\beta_k |\varphi|^{2\sigma_k}(t,x)
+ \lambda \left(U \ast |\varphi(t,\cdot)|^2\right)(x)
-\Omega . R \right)
\varphi(t,x),
$$
but we restrict ourselves to $K=1$ for the sake of simplicity.
Equation \eqref{eq:GPE} is hamiltonian for the energy functional
\begin{equation}
\label{eq:NRJGPE}
  E(\varphi) =  \int_{\R^d} \left( \frac{1}{4} 
\|\nabla \varphi\|^2 
 + \frac{1}{2} V |\varphi|^2 
 + \frac{\beta}{2\sigma+2} |\varphi|^{2\sigma+2}
  + \frac{\lambda}{4} (U \ast |\varphi|^2) |\varphi|^2
  - \frac{\Omega}{2} \overline{\varphi} R \varphi \right) dx,
\end{equation}
provided $U$ is a real-valued convolution kernel, symmetric with respect to the origin \footnote{For real-valued functions, symmetry with respect to the origin is equivalent to real-valued Fourier transform.}. 
In practice the convolution kernel $U$ may for example correspond to a
Poisson equation ($U(x)=1/(4\pi|x|)$ in dimension $d=3$) or it may
represent dipole-dipole interactions (see \cite{becrd}). 

The main goal of this paper is the analysis of numerical methods for the time integration of \eqref{eq:GPE} that preserve the energy \eqref{eq:NRJGPE} or a discretized analogue of it. In particular, we are interested in the order (in time) of such methods. A well known method is the Crank-Nicolson method introduced in \cite{CN47} for parabolic problems (see for example a posteriori error estimates in \cite{Akrivis2006})
and applied in \cite{DFP81} to Schr\"o\-din\-ger equations.
For all nonlinearities, these methods are fully implicit.
However, they have second order in time (see \cite{SS1984} for the case of the
cubic NLS equation and \cite{WANG2013670} for the case of a system with
possibly fractional derivatives) and preserve discrete analogues
of the energy \eqref{eq:NRJGPE} as well as the total mass (squared $L^2$-norm)
of the solutions. Unfortunately, since they are fully implicit, these
methods are costly. To work around this problem, the methods introduced and
analyzed in this paper belong to the family of relaxation methods.

Relaxation methods for Schr\"odinger equations were introduced in
\cite{BessePhD, Besse2004}. They have been applied to different NLS equation for example in the context of plasma physics \cite{Trabelsi2014}. 
For cubic nonlinearities ($\sigma=1$ in \eqref{eq:GPE}),
they are linearly implicit hence very popular \cite{AnBeKl2011,DaOw09,XavRom2015,Ga14,HeWa18}.
They preserve the $L^2$-norm and a discrete analogue of
\eqref{eq:NRJGPE}.
It is well-known that they have numerical order 2 but up to our knowledge
there is no proof of order 2 in the literature.
This paper presents two new results with respect to relaxation methods
for \eqref{eq:GPE}.
First, we prove rigorously that the classical relaxation method,
applied to the classical cubic NLS equation
({\it i.e.} \eqref{eq:GPE} with $V \equiv 0$, $\beta=1$, $\sigma=1$,
$\lambda=0$ and $\Omega=0$) is of order 2.
Second, we present a generalized relaxation method that allows to deal
with general power law nonlinearities ($\sigma \neq 1$ in \eqref{eq:GPE}) and with the full GPE.
The generalized relaxation methods that we introduce in this context
are implicit (actually explicit for $\sigma \leq 4$), have numerical order 2 and we show that they preserve
an energy which is also a discretized analogue of \eqref{eq:NRJGPE}. 

%
%
%

This paper is organized as follows. In Section \ref{sec:pres} we recall the definition of the Crank-Nicolson method and give a short explanation of its energy preserving property.
Section \ref{sec:relax} is devoted to relaxation methods applied to \eqref{eq:GPE}: In a first part we recall the method introduced in \cite{BessePhD} for the cubic Schr\"odinger equation and in a second part we give a proof of the optimal order of convergence  for an initial datum belonging to $H^{s+4}(\R^d)$, $d$ in $\{1,2,3\}$ and $s>d/2$. In section \ref{sec:genrelax} we propose a generalized relaxation method that allows to deal with general nonlinearites and we prove that this method is also an energy preserving method. Section \ref{sec:num} deals with numerical results in different physical models showing the efficiency of the methods.

\section{Preservation of energy and Crank-Nicolson scheme \label{sec:pres}}
A usual way to prove the conservation of the energy \eqref{eq:NRJGPE}
consists in multiplying the equation \eqref{eq:GPE} by
$\overline{\partial_t \varphi}(t,x)$,
where $\bar{z}$ denotes the conjugate value of a complex $z$,
integrating over $\R^d$ and taking the real part of the result.
This computation relies on the identity which holds for all smooth
functions $\varphi$ of time with values into a space of sufficiently integrable
functions
\begin{equation}
  \label{eq:id_ener}
  \mathrm{Re} \int_{\R^d} |\varphi|^{2\sigma}\varphi \overline{\partial_t \varphi}\, dx = \frac{1}{2\sigma+2} \frac{d}{dt}  \int_{\R^d} |\varphi|^{2\sigma+2}\, dx,\quad \forall \sigma\geq 0.
\end{equation}
A possible way to derive numerical schemes that preserve an energy
functional is therefore to mimic this identity at the discrete level.

In 1981, Delfour, Fortin and Payre \cite{DFP81},
following an idea of Strauss and Vasquez \cite{StrVas78},
proposed a way to deal with the nonlinear term $|\varphi|^{2\sigma}\varphi$
for the Crank-Nicolson scheme.
This method generalizes the second order mid-point scheme for the
linear Schr\"odinger equation
\[
  i\frac{\varphi_{n+1}-\varphi_n}{\delta t}=-\frac{1}{2}\Delta \frac{\varphi_{n+1}+\varphi_n}{2},\]
where $\varphi_n(x)$ denotes an approximation of $\varphi(t_n,x)$
with the discrete time $t_n=n\delta t $ defined with the time step $\delta t$.
Their approach can be explained as follows.
If one looks for a real-valued function $g:\C^2\rightarrow \R$ such that
the scheme takes the form
\begin{equation}
  \label{eq:methCNoupas}
i\frac{\varphi_{n+1}-\varphi_n}{\dt}
 = \left( -\frac{1}{2} \Delta 
+ V 
+ {\beta}g(\varphi_n,\varphi_{n+1})
+ \lambda \left(U\ast\left(\frac{|\varphi_{n+1}|^2+|\varphi_n|^2}{2}\right)\right)
- \Omega.R
\right)\frac{\varphi_{n+1}+\varphi_n}{2},
\end{equation}
then, multiplying this relation by $i\overline{\varphi_{n+1}-\varphi_n}$, integrating overs $\R^d$ 
and taking the real part, as we did in the time-continuous setting above,
yields to $0$ in the left-hand side and several terms in the right-hand side.
Amongst these terms, those involving $g$ are equal to
\begin{equation*}
  \beta \int_{\R^d} g(\varphi_n,\varphi_{n+1})\left(|\varphi_{n+1}|^2-|\varphi_n|^2\right) dx,
\end{equation*}
since $g$ is real-valued.
Let us denote by $G$ the function $v\mapsto |v|^{2\sigma+2}/(2\sigma+2)$.
A sufficient condition for the method \eqref{eq:methCNoupas}
to preserve an energy of the form \eqref{eq:NRJGPE} is therefore to have
\begin{equation*}
  g(\varphi_n,\varphi_{n+1})\left(|\varphi_{n+1}|^2-|\varphi_n|^2\right)
    =G(\varphi_{n+1})-G(\varphi_n).
\end{equation*}
This is exactly the definition of $g$ chosen in \cite{DFP81}.


In the following, the Crank-Nicolson method for the GPE \eqref{eq:GPE}
is therefore defined using the formula
\begin{eqnarray}
\lefteqn{i\frac{\varphi_{n+1}-\varphi_n}{\dt}} & \label{eq:CNmethod}\\
 = &\displaystyle \left( -\frac{1}{2} \Delta 
+ V 
+ \frac{\beta}{\sigma+1}
 \frac{|\varphi_{n+1}|^{2\sigma+2} - |\varphi_n|^{2\sigma+2}}{
|\varphi_{n+1}|^2-|\varphi_n|^2}
+ \lambda \left(U\ast\left(\frac{|\varphi_{n+1}|^2+|\varphi_n|^2}{2}\right)\right)
- \Omega.R
\right)\frac{\varphi_{n+1}+\varphi_n}{2}. \nonumber
\end{eqnarray}
We shall use the notation
\begin{equation*}
  \varphi_{n+1} = \Phi_{\dt}^{\rm CN} (\varphi_n),
\end{equation*}
for the Crank-Nicolson method \eqref{eq:CNmethod}.
In the expression above, the term corresponding to the nonlinearity
should be understood as
\begin{equation}
  \label{eq:jesaispas}
  \frac{\beta}{\sigma+1}\frac{|\varphi_{n+1}|^{2\sigma+2} - |\varphi_n|^{2\sigma+2}}{
|\varphi_{n+1}|^2-|\varphi_n|^2} =
\frac{\beta}{\sigma+1}
\sum_{k=0}^\sigma
|\varphi_{n+1}|^{2k} |\varphi_n|^{2(\sigma-k)},
\end{equation}
so that it is indeed non-singular and it is consistent with the non-linear
term $\beta |\varphi|^{2\sigma}$.
The Crank-Nicolson method is fully implicit.
It is known to have order two for the cubic NLS equation \cite{SS1984}.
Moreover it preserves exactly the $L^{2}$-norm of the solution
as well as the following energy:
\begin{equation}
\label{eq:NRJCN}
  E_{\rm CN}(\varphi)  = E(\varphi),
\end{equation}
with $E$ defined by \eqref{eq:NRJGPE}.

In Section \ref{sec:genrelax}, we shall use similar ideas to derive
energy-preserving relaxation methods for general power laws nonlinearities.
Before doing so, we first deal with the classical relaxation method
in Section \ref{sec:relax}.

\section{The classical relaxation method}
\label{sec:relax}
\subsection{An energy preserving method}
In \cite{Besse2004}, Besse introduced the usual relaxation method
\eqref{eq:relaxclassic} 
applied to the nonlinear Schr\"odinger equation
\eqref{eq:GPE} with $V=0$, $\lambda=0$, $\Omega=0$ and $\sigma=1$ that
is known as the cubic nonlinear Schr\"odigner equation
\begin{equation}
  \label{eq:cubicnls}
  i\partial_t\varphi(t,x) = -\frac{1}{2} \Delta \varphi(t,x)
+ \beta |\varphi(t,x)|^{2} \varphi(t,x),
\end{equation}
with $\varphi(0,x)=\varphi_{\text{in}}(x)$. The idea of the relaxation method is
to add to \eqref{eq:cubicnls} a new unknown $\Upsilon=|\varphi|^2$ and
the equation \eqref{eq:cubicnls} is transformed in
\begin{equation}
  \label{eq:relax_cont}
  \left \{
\begin{array}{l}
  \Upsilon = |\varphi(t,x)|^{2},\\
\displaystyle   i\partial_t\varphi(t,x) = -\frac{1}{2} \Delta \varphi(t,x)+ \beta \Upsilon \varphi(t,x).
\end{array}\right .
\end{equation}
The relaxation method then consists in discretizing both
equations respectively at discrete times $t_n$ and $t_{n+1/2}$ and to
solve iteratively 
\begin{equation}
\label{eq:relaxclassic}
  \left\{
    \begin{array}{l}
      \displaystyle\frac{\Upsilon_{n+1/2}+\Upsilon_{n-1/2}}{2} = |\varphi_n|^2,\\
      \displaystyle i\frac{\varphi_{n+1}-\varphi_n}{\dt}
 = \left( -\frac{1}{2} \Delta 
+ \beta \Upsilon_{n+1/2} 
\right)\frac{\varphi_{n+1}+\varphi_n}{2},
    \end{array}
  \right.
\end{equation}
to compute approximations $\varphi_n$ of $\varphi(n\dt)$. This system
is usually initialized with $\Upsilon_{-1/2}=|\varphi(-\dt/2)|^2$ or
by second order approximation of $|\varphi(-\dt/2)|^2$. This method is
linearly implicit (recall that $\sigma=1$). Moreover, 
it is known to preserve exactly the $L^2$-norm and the
discrete energy \cite{Besse2004}:
\begin{equation}
  \label{eq:ener_for_relax}
  E_{\mathrm{rlx}}(\varphi,\Upsilon) = \frac{1}{4} \int_{\R^d} \|\nabla \varphi\|^2 dx
+ \frac{\beta}{2} \int_{\R^d} \Upsilon |\varphi|^2  dx
- \frac{\beta}{4}  \int_{\R^d} \Upsilon^2 dx.
\end{equation}
Indeed
\[  \displaystyle \mathrm{Re}\left ( \Upsilon_{n+1/2}
  \frac{\varphi_{n+1}+\varphi_n}{2}\overline{\varphi_{n+1}-\varphi_n}\right )=\frac{\Upsilon_{n+1/2}}{2}\Big(|\varphi_{n+1}|^2-|\varphi_{n}|^2\Big).\]
But
\[
\begin{array}{ll}
  \displaystyle \Upsilon_{n+1/2}\Big(|\varphi_{n+1}|^2-|\varphi_{n}|^2\Big)& \displaystyle =\Upsilon_{n+1/2}\Big(|\varphi_{n+1}|^2-|\varphi_{n}|^2\Big)+\Upsilon_{n-1/2}\Big(|\varphi_{n}|^2-|\varphi_{n}|^2\Big) \\
  & = \displaystyle \Big(\Upsilon_{n+1/2}|\varphi_{n+1}|^2-\Upsilon_{n-1/2}|\varphi_{n}|^2\Big)-{\Big(\Upsilon_{n+1/2} -\Upsilon_{n-1/2}\Big)|\varphi_{n}|^2}.
\end{array}
\]
Using the definition of $\Upsilon_{\cdot+1/2}$ in \eqref{eq:relaxclassic}, a simple computation
leads to
\[
\displaystyle {\Big(\Upsilon_{n+1/2}
  -\Upsilon_{n-1/2}\Big)|\varphi_{n}|^2} =
\frac{\Big(\Upsilon_{n+1/2}\Big)^2-\Big(\Upsilon_{n-1/2}\Big)^2}{2}.
\]
We therefore conclude that
\[
  \mathrm{Re}\left ( \Upsilon_{n+1/2}
  \frac{\varphi_{n+1}+\varphi_n}{2}\overline{\varphi_{n+1}-\varphi_n}\right )=\Big(\Upsilon_{n+1/2}|\varphi_{n+1}|^2-\Upsilon_{n-1/2}|\varphi_{n}|^2\Big)+\frac{\Big(\Upsilon_{n+1/2}\Big)^2-\Big(\Upsilon_{n-1/2}\Big)^2}{2},
\]
which allows to prove the conservation of the discrete energy
\eqref{eq:ener_for_relax}.

It is interesting to note the consistency of the energy associated to
relaxation scheme with the energy \eqref{eq:NRJGPE} for cubic
nonlinear Schr\"odinger equation
\[
E_{\mathrm{rlx}}(\varphi,|\varphi|^2)=E(\varphi).
\]
The relaxation method was proved to converge in \cite{Besse2004} but
consistency analysis was missing. We present it in the next subsection.
\subsection{Consistency analysis for NLS equation with cubic nonlinearity}

The aim of this subsection is to prove that this method has temporal order
2 under fairly general assumptions. This fact is supported by numerical
evidences in the literature for years. We provide the first rigorous proof
below.


The first equation in \eqref{eq:relaxclassic} is the discrete equivalent of
the continuous constraint $\Upsilon=|\varphi|^2$.
In particular, this constraint is {\it not} an evolution equation.
Therefore, we use the ideas introduced in \cite{Besse2004} and rewrite the
continuous equation \eqref{eq:GPE} (recall that $V=0$, $\lambda=0$ and
$\Omega=0$ and $\sigma=1$) as the system
\begin{equation}
  \label{eq:equivrelax}
  \left \{
    \begin{array}{l}
      i\partial_t \varphi +\dfrac{1}{2}\Delta \varphi = \beta \Upsilon \varphi, \\
\partial_t \Upsilon = 2 \mathrm{Re} (\overline{v} \varphi), \\
      i\partial_t v +\dfrac{1}{2}\Delta v = \beta(\partial_t \Upsilon \varphi + \Upsilon \partial_t \varphi).
    \end{array}
\right .
\end{equation}
The discrete system \eqref{eq:relaxclassic}  has a discrete augmented equivalent
(see \cite{BessePhD, Besse2004}).
Let us denote by $v_{n+\frac12}=\dfrac{\varphi_{n+1}-\varphi_{n}}{\dt}$
the discrete time derivative of $\varphi_n$ and define the nonlinearities as
\begin{equation}
  \label{eq:method}
\left \{\begin{array}{l}
\displaystyle \Phi_{n+\frac12}=\Upsilon_{n+\frac12} \left(\frac{\varphi_{n+1}+\varphi_{n}}{2}\right),\\[0.2cm]
\displaystyle \Xi_{n+\frac12}=2 \mathrm{Re}
\left(v_{n+\frac12}\left(\overline{\frac{\varphi_{n+1}+\varphi_{n}}{2}}\right)\right) ,\\[0.2cm]
\displaystyle V_{n+\frac12}=\left(\frac{\Upsilon_{n+\frac32}+\Upsilon_{n-\frac12}}{2}\right)\left(\frac{v_{n+\frac32}+2v_{n+\frac12}+v_{n-\frac12}}{4}\right)\\[0.3cm]
\displaystyle \qquad \qquad +2\mathrm{Re}
\left(v_{n+\frac12}\left(\overline{\frac{\varphi_{n+1}+\varphi_{n}}{2}}\right)\right)
\left(\frac{\varphi_{n+2}+\varphi_{n+1}+\varphi_{n}+\varphi_{n-1}}{4}\right).
\end{array}  
\right .
\end{equation}
The augmented system writes
\begin{equation}
\label{systran}
\left \{
\begin{array}{lr}
\displaystyle  i \frac{\varphi_{n+2}-\varphi_{n+1}}{\dt}+\dfrac{1}{2}\Delta \left(\frac{\varphi_{n+2}+\varphi_{n+1}}{2}
\right) = \beta \Phi_{n+\frac32}, &(\ref{systran}.b)\\[0.3cm]
\displaystyle \frac{\Upsilon_{n+\frac32}-\Upsilon_{n-\frac12}}{2 \dt}=\Xi_{n+\frac12},&  (\ref{systran}.a)\\[0.3cm]
\displaystyle i \frac{v_{n+\frac32}-v_{n-\frac12}}{2\dt}+\dfrac{1}{2}\Delta\left(\frac{v_{n+\frac32}+2
v_{n+\frac12}+v_{n-\frac12}}{4}\right)=\beta V_{n+\frac12}.& (\ref{systran}.c)
\end{array}
\right .
\end{equation}
This system allows to compute $(\varphi_{n+2},\Upsilon_{n+\frac32},v_{n+\frac32})$
from 
$$ X_n :=
\left(\Upsilon_{n-\frac12},\Upsilon_{n+\frac12},\varphi_{n-1},\varphi_{n},\varphi_{n+1},
v_{n-\frac12},v_{n+\frac12}\right).$$ 
We consider the system \eqref{systran} as the mapping 
\begin{equation*}
  X_n\mapsto X_{n+1}.
\end{equation*}
The seven variables involved in $X_n$ are not independent.
If they satisfiy the five relations
\begin{equation*}
  \left \{
    \begin{array}{l}
      \displaystyle v_{n-\frac12}=\frac{\varphi_{n}-\varphi_{n-1}}{\dt},\\
      \displaystyle v_{n+\frac12}=\frac{\varphi_{n+1}-\varphi_{n}}{\dt},\\
      \displaystyle \Upsilon_{n+\frac12}+\Upsilon_{n-\frac12}=2|\varphi_n|^2,\\
      \displaystyle i\frac{\varphi_{n}-\varphi_{n-1}}{\dt} = \left( -\dfrac{1}{2}\Delta + \beta\Upsilon_{n-1/2} 
\right)\frac{\varphi_{n}+\varphi_{n-1}}{2},\\
      \displaystyle i\frac{\varphi_{n+1}-\varphi_n}{\dt} = \left(-\dfrac{1}{2} \Delta + \beta\Upsilon_{n+1/2} 
\right)\frac{\varphi_{n+1}+\varphi_n}{2},
    \end{array}
\right .
\end{equation*}
then the seven variables in $X_{n+1}$ satisfy the same five relations with
$n$ replaced by $n+1$. This fact is proved in \cite{Besse2004}.
We describe now how to build the seven initial data in $X_0$ from $\varphi_0=\varphi_{\text{in}}$
and $\Upsilon_{-\frac12}$ so that they satisfy the five relations above:
\begin{equation}
\label{eq:IDrelext}
  \left \{
    \begin{array}{ll}
      \displaystyle \Upsilon_{\frac12}&\displaystyle =2|\varphi_0|^2-\Upsilon_{-\frac12},\\
      \displaystyle \varphi_{-1}&\displaystyle = \left(2-i\dt \left(-\dfrac{1}{2}\Delta + \beta\Upsilon_{-1/2}\right)\right )^{-1}  \left(2+i\dt \left(-\dfrac{1}{2}\Delta + \beta\Upsilon_{-1/2}\right)\right ) \varphi_0,\\
      \displaystyle \varphi_{1}&\displaystyle = \left(2+i\dt \left(-\dfrac{1}{2}\Delta + \beta\Upsilon_{1/2}\right)\right )^{-1}  \left(2-i\dt \left(-\dfrac{1}{2}\Delta + \beta\Upsilon_{1/2}\right)\right ) \varphi_0,\\
      \displaystyle v_{-\frac12}&\displaystyle =(\varphi_{0}-\varphi_{-1})/\dt,\\
      \displaystyle v_{\frac12}& \displaystyle =(\varphi_{1}-\varphi_{0})/\dt.
    \end{array}
\right .
\end{equation}
Let us define the operators
$$
A=(i-\dt\Delta/4)^{-1}(i+\dt\Delta/4) \text{ and }
B=(i-\dt\Delta/4)^{-1},
$$
and the matrix of operators
  \begin{equation*}
    \mathcal{C} =
    \begin{pmatrix}
      I & 0 & 0 & 0 & 0 & 0 & 0 \\
      0 & I & 0 & 0 & 0 & 0 & 0 \\
      0 & 0 & B & 0 & 0 & 0 & 0 \\
      0 & 0 & 0 & B & 0 & 0 & 0 \\
      0 & 0 & 0 & 0 & B & 0 & 0 \\
      0 & 0 & 0 & 0 & 0 & B & 0 \\
      0 & 0 & 0 & 0 & 0 & 0 & B
    \end{pmatrix}.
  \end{equation*}
The mapping $X_n\mapsto X_{n+1}$ reads
\begin{equation}
\label{eq:relext}
\left ( 
  \begin{array}{c}
    \Upsilon_{n+\frac12}\\
    \Upsilon_{n+\frac32}\\
    \varphi_n \\
    \varphi_{n+1} \\
    \varphi_{n+2} \\
    v_{n+\frac12}\\
    v_{n+\frac32}\\
  \end{array}
\right )
=
\begin{pmatrix}
  0 & I & 0 & 0 & 0 & 0 & 0 \\
  I & 0 & 0 & 0 & 0 & 0 & 0 \\
  0 & 0 & 0 & I & 0 & 0 & 0 \\
  0 & 0 & 0 & 0 & I & 0 & 0 \\
  0 & 0 & 0 & 0 & A & 0 & 0 \\
  0 & 0 & 0 & 0 & 0 & 0 & I \\
  0 & 0 & 0 & 0 & 0 & A & A-I
\end{pmatrix}
\left ( 
  \begin{array}{c}
    \Upsilon_{n-\frac12}\\
    \Upsilon_{n+\frac12}\\
    \varphi_{n-1} \\
    \varphi_{n} \\
    \varphi_{n+1} \\
    v_{n-\frac12}\\
    v_{n+\frac12}\\
  \end{array}
\right )
+\dt \mathcal{C}
\left ( 
  \begin{array}{c}
    0\\
    2\Xi_{n+\frac12}\\
    0\\
    0\\
    \beta \Phi_{n+\frac32}\\
    0\\
    2\beta  V_{n+\frac12}
  \end{array}
\right ).
\end{equation}
In a more compact form, we define $\mathcal{B}$ and $\mathcal{M}$ so that
the mapping \eqref{eq:relext} reads
\begin{equation}
  \label{eq:relextsimpl}
  X_{n+1}= \mathcal{B} X_n +\dt \mathcal{C} \mathcal{M}(X_n,X_{n+1}).
\end{equation}
We introduce the hypotheses that will allow us to prove our consistency
and convergence result for the classical relaxation method
\eqref{eq:relaxclassic} in Theorem \ref{th:machinbidule}.

\begin{remark}
	The results below extend to more general cases. In particular, one may treat the case where $V$ is non zero smooth autonomous potential such that the multiplication by $V$ is a bounded linear operator between Sobolev spaces and the case where $V$ is a nonautonomous such operator with sufficient regularity with respect to time.  
\end{remark}

\begin{hypo} \label{hypobornes}
  We fix $d \in \{1,2,3\}$ and $s>d/2$.
  We assume $\varphi_0 \in H^{s+4}(\R^d)$ is given. We denote by $T^*>0$
  the existence time of the maximal solution $\varphi$ of the Cauchy problem
  \eqref{eq:GPE} (with $V=0$, $ \lambda=0$, $\Omega=0$ and $\sigma=1$) in $H^{s+4}(\R^d)$.
  We assume there exists $\dt_0>0$ such that $\tau \mapsto \varphi(\tau,\cdot)$ is a smooth map from $(-\dt_0,T^*)$ to $H^{s+4}(\R^d)$.
  Moreover, we assume that there exists $R_1, R_2>0$ such that for all
  $\Upsilon_{-1/2}\in H^{s+4}(\R^d)$ with $\|\Upsilon_{-1/2}\|_{H^{s+4}} \leq R_1$,
  the numerical solution $X_n$ (with initial datum \eqref{eq:IDrelext})
  is uniquely determined by \eqref{eq:relext} for all $\dt\in(0,\dt_0)$
  and all $n\in\N$ such that $nh\leq T$, and it satisfies
  $\|X_n\|_{(H^{s+2}(\R^d))^7}\leq R_2$ for all such $n$.  
\end{hypo}

\begin{remark}
  The hypotheses above on the exact solution are fullfilled in several cases. For example, for the exact solution $\varphi$, it is well known (see \cite{GV79}) that $T^*=+\infty$ in at least two cases:
  \begin{itemize}
  \item if $\varphi_{\text{in}}$ has small $H^{s+4}$-norm and $\beta<0$
  \item if $\beta>0$ and $\varphi_{\text{in}}\in H^{s+4}(\R^d)$.
  \end{itemize}
For the numerical solution, they are fullfilled provided $T^*<+\infty$ (see \cite{Besse2004}) and also when $T^*=+\infty$ and $\beta>0$ (see \cite{BessePhD}).
\end{remark}
Let $t_n=n\dt$ denote the discrete times and $t\mapsto X(t)$ the vector 
\begin{equation}
  \label{eq:defsolex}
X(t)=  \begin{pmatrix}
|\varphi(t-\dt/2,\cdot)|^2\\    
|\varphi(t+\dt/2,\cdot)|^2\\
\varphi(t-\dt,\cdot)\\
\varphi(t,\cdot)\\
\varphi(t+\dt,\cdot)\\
\partial_t \varphi(t-\dt/2,\cdot)\\
\partial_t \varphi(t+\dt/2,\cdot)
  \end{pmatrix}.
\end{equation}
Using the definition of $\mathcal{M}$
(see \eqref{eq:method} and \eqref{eq:relextsimpl}),
the fact that $H^s(\R^d)$, is an algebra since $s>d/2$, and the fact that
the exact and numerical solutions stay in a bounded set of $H^s(\R^d)$,
it is easy to prove the following lemma.
\begin{lemma}\label{lem:method}
Assume Hypotheses \ref{hypobornes} is satisfied.
There exists $C>0$ such that for all $\dt \in (0,\dt_0)$,
all $n \in \mathbb{N}$ such that $(n+1)\dt \leq T$
\begin{eqnarray*}
\lefteqn{\| \mathcal{M}(X_n,X_{n+1})-\mathcal{M}(X(t_n),X(t_{n+1}))\|_{(H^s(\R^d))^7}}
\\
&\leq &C\left ( \|X_n-X(t_n)\|_{(H^s(\R^d))^7}+\|X_{n+1}-X(t_{n+1})\|_{(H^s(\R^d))^7}\right ).
\end{eqnarray*}
Note that the constant $C$ depends only on the initial data.
\end{lemma}

\noindent
Before starting the proof of our main result of this section (see Theorem
\ref{th:machinbidule}), we state and prove
another lemma.

\begin{lemma}\label{lem:puissancebloc}
There exists a constant $b>0$ such that
the operator $\mathcal{B}$ defined in \eqref{eq:relext}
and \eqref{eq:relextsimpl} satisfies
for all
$\dt>0$, and $n\in \N$, 
  \begin{equation}
    \label{eq:2}
    \vvvert \mathcal{B}^n\mathcal{C}\vvvert \leq b,
  \end{equation}
where $\vvvert \cdot \vvvert$ is the norm of linear continuous operators from $(H^s(\R^d))^7$ to itself. 
\end{lemma}
\begin{proof}
  The operator $\mathcal{B}$ is defined by three diagonal blocks
\begin{equation*}
\mathcal{B}_1=
\begin{pmatrix}
  0 & I \\ I&0
\end{pmatrix}, \quad
\mathcal{B}_2=
\begin{pmatrix}
  0 & I & 0\\
0& 0& I\\
0& 0& A
\end{pmatrix}, \quad \text{and} \quad
\mathcal{B}_3=
\begin{pmatrix}
  0 & A \\ A & A-I
\end{pmatrix}.
\end{equation*}
The first block $\mathcal{B}_1$ is an isometry from $(H^s(\R^d))^2$ to itself and so are all its powers. The powers of the second block $\mathcal{B}_2$ read for all $n\geq 2$
\begin{equation*}
  \mathcal{B}_2^n=
\begin{pmatrix}
  0 & 0 & A^{n-2}\\
0& 0& A^{n-1}\\
0& 0& A^n
\end{pmatrix}.
\end{equation*}
Since all the powers of $A$ are of norm less than one, we infer that the norm of $\mathcal{B}_2^n$ is less than $\sqrt{3}$. Since the norm of $B$ is less than one, we infer that the norm of 
$$\mathcal{B}_2^n
\begin{pmatrix}
  B&0&0\\0&B&0\\0&0&B
\end{pmatrix},
$$
is less than $\sqrt{3}$.

The block $\mathcal{B}_3$ can be diagonalized by blocks as 
\begin{equation*}
  \mathcal{B}_3=
  \begin{pmatrix}
    I & I \\ -I & A
  \end{pmatrix}
  \begin{pmatrix}
    -I & 0 \\ 0 & A
  \end{pmatrix}
  \begin{pmatrix}
    I & I \\ -I & A
  \end{pmatrix}^{-1},
\end{equation*}
so that for all $n \in \N$, 
\begin{equation}
\label{eq:puissanceB3}
  \mathcal{B}_3^n=
  \begin{pmatrix}
    I & I \\ -I & A
  \end{pmatrix}
  \begin{pmatrix}
    (-I)^n & 0 \\ 0 & A^n
  \end{pmatrix}
  \begin{pmatrix}
    I & I \\ -I & A
  \end{pmatrix}^{-1}.
\end{equation}
Therefore, the last $2\times 2$ block of $\mathcal{B}^n\mathcal{C}$ reads
\begin{equation*}
  \mathcal{B}_3^n
  \begin{pmatrix}
    B & 0 \\
    0 & B \\
  \end{pmatrix}
  =
  \begin{pmatrix}
    ((-I)^nA + A^n)B(I+A)^{-1} & (A^n-(-I)^n)B(I+A)^{-1}\\
    ((-I)^{n+1}A + A^{n+1})B(I+A)^{-1} & ((-I)^n+A^{n+1})B(I+A)^{-1}
  \end{pmatrix}.
\end{equation*}
Since $B(I+A)^{-1}=(1/2)I$, we infer
\begin{equation*}
  \mathcal{B}_3^n
  \begin{pmatrix}
    B & 0 \\
    0 & B \\
  \end{pmatrix}
  =\frac12
  \begin{pmatrix}
    ((-I)^nA + A^n) & (A^n-(-I)^n)\\
    ((-I)^{n+1}A + A^{n+1}) & ((-I)^n+A^{n+1})
  \end{pmatrix}.
\end{equation*}
Since all the powers of $A$ are of norm less than $1$, we infer that
\begin{equation*}
  \forall n\in\N,\qquad \left\vvvert \mathcal{B}_3^n
  \begin{pmatrix}
    B & 0 \\
    0 & B \\
  \end{pmatrix}
  \right\vvvert \leq 4.
\end{equation*}
This proves the result.
\end{proof}

\begin{remark}
The powers of the operator $\mathcal{B}$ are not uniformly bounded. However, the powers of $\mathcal{B}$ mutiplied by $\mathcal{C}$ are uniformly bounded as shown above. The main reason is that the third matrix in the right hand side of \eqref{eq:puissanceB3} becomes singular at the end of the spectrum of $\Delta$.
\end{remark}

\begin{lemma}
\label{lem:estiminit}
Assume $d, s, R_1, \dt_0$ and $\varphi_{\text{in}}$ are given as in Hypotheses \ref{hypobornes}.
There exists $c>0$ such that for all $\Upsilon_{-1/2}\in H^{s+4}$ with $\|\Upsilon_{-1/2}\|_{H^{s+4}}\leq R_1$ and all
$\dt\in(0,\dt_0)$,
\begin{equation}
\label{eq:estimDI}
  \|X_0-X(0)\|_{(H^{s+2}(\R^d))^7} \leq
c \left(\|\Upsilon_{-1/2}-|\varphi(-\dt/2)|^2\|_{H^{s+4}(\R^d)} +\dt^2\right).
\end{equation}
\end{lemma}

\begin{proof}
  The $H^{s+2}$-norm of each of the seven components of the vector $X_0-X(0)$
  is estimated separately.
  The $H^{s+2}$-norm of the first component
  $\Upsilon_{-1/2}-|\varphi(-\dt/2)|^2$ is bounded by the $H^{s+4}$-norm of the same quantity.
  For the $H^{s+2}$-norm of the second component, we define
  $f(t)=|\varphi(t)|^2$, which is a smooth function from $(-\dt_0,\dt_0)$ to
  $H^{s+4}$, thanks to Hypotheses \ref{hypobornes}. We may write using
  a Taylor formula at $0$
  \begin{equation*}
    f\left(-\frac{\dt}{2}\right) = |\varphi_0|^2 - 2\mathrm{Re}\left(\overline{\varphi_0}\partial_t\varphi(0)\right)\frac{\dt}{2} + \int_0^{-\dt/2}\left(\frac{-\dt}{2}-\sigma\right) f''(\sigma){\rm d}\sigma,
  \end{equation*}
  and similarly
  \begin{equation*}
    f\left(\frac{\dt}{2}\right) = |\varphi_0|^2 + 2\mathrm{Re}\left(\overline{\varphi_0}\partial_t\varphi(0)\right)\frac{\dt}{2} + \int_0^{\dt/2}\left(\frac{\dt}{2}-\sigma\right) f''(\sigma){\rm d}\sigma.
  \end{equation*}
  Therefore, one has
  \begin{eqnarray*}
    \Upsilon_{1/2}-f\left(\frac{\dt}{2}\right) & = & 2|\varphi_0|^2 - \Upsilon_{-1/2} - f\left(\frac{\dt}{2}\right)\\
                                               & = & f\left(-\frac{\dt}{2}\right) - \int_0^{-\dt/2}\left(\frac{-\dt}{2}-\sigma\right) f''(\sigma){\rm d}\sigma  -\int_0^{\dt/2}\left(\frac{\dt}{2}-\sigma\right) f''(\sigma){\rm d}\sigma - \Upsilon_{-1/2}.
  \end{eqnarray*}
  By triangle inequality, We infer that
  \begin{equation}
    \label{eq:estim2eme}
    \|\Upsilon_{1/2}-|\varphi(\dt/2)|^2\|_{H^{s+2}} \leq \|\Upsilon_{1/2}-|\varphi(\dt/2)|^2\|_{H^{s+4}} \leq c
    \left(\|\Upsilon_{-1/2}-|\varphi(-\dt/2)|^2\|_{H^{s+4}} + \dt^2\right),
  \end{equation}
  where $c=\max(1,\sup_{\sigma\in(-\dt_0),\dt_0}\|f''(\sigma)\|_{H^{s+2}})$
  only depends on the exact solution of \eqref{eq:GPE}.
  We now estimate the $H^{s+2}$-norm of the fifth component of $X_0-X(0)$.
  Note that the $H^{s+2}$-norm of the third component can be estimated the
  very same way and that the $H^{s+2}$-norm of the fourth component is zero.
  We start with the identity
  \begin{equation*}
    \varphi_1=\varphi_0
    - i\dt\left(-\frac12\Delta+\beta\Upsilon_{1/2}\right)\frac{\varphi_1+\varphi_0}{2},
  \end{equation*}
  and we denote by $r(\dt)$ the consistency error defined by
  \begin{equation*}
    r(\dt)= \varphi(\dt)-\varphi_0+i\dt\left(-\frac12\Delta + \beta\Upsilon_{1/2}\right)\frac{\varphi(\dt)+\varphi_0}{2}.
  \end{equation*}
  Using the fact that $t\mapsto \varphi(t, \cdot)$ is a smooth function from
  $(-\dt_0,\dt_0)$ to $H^{s+4}$, we may write another Taylor expansion to obtain
  \begin{eqnarray*}
    \lefteqn{r(\dt)}\\\
    & = & \varphi_0+\dt\partial_t\varphi(0) +\frac{\dt^2}{2}\partial_t^2 \varphi(0)+ \int_0^{\dt} \dfrac{(\dt-\sigma)^2}{2}\partial_t^3\varphi(\sigma){\rm d}\sigma - \varphi_0 \\
    & &+ i\dt\left(
          -\frac12\Delta+\beta\Upsilon_{1/2}\right)
          \left(
          \varphi_0+\dfrac{\dt}{2}\partial_t\varphi(0)+\frac12\int_0^{\dt}(\dt-\sigma)\partial_t^2\varphi(\sigma){\rm d}\sigma
          \right)\\
    &=& -i\dt^2\beta\mathrm{Re}\left(\dfrac{i}{2}\overline{\varphi_0}\Delta\varphi_0\right)\varphi_0 +i\dt\beta\left(|\varphi_0|^2-\Upsilon_{-1/2}\right)\left(\varphi_0-i\dfrac{\dt}{2}\left(-\dfrac{1}{2}\Delta+\beta|\varphi_0|^2\right)\varphi_0\right)\\
    & & + i \dfrac{\dt}{2}\left( -\frac12\Delta+\beta\Upsilon_{1/2}\right)\int_0^{\dt}(\dt-\sigma)\partial_t^2\varphi(\sigma){\rm d}\sigma+ \int_0^{\dt} \dfrac{(\dt-\sigma)^2}{2}\partial_t^3\varphi(\sigma){\rm d}\sigma \\
    &=& -i\dt^2\beta\mathrm{Re}\left(\dfrac{i}{2}\overline{\varphi_0}\Delta\varphi_0\right)\varphi_0 +i\dt\beta\left(|\varphi_0|^2-|\varphi\left(-\dt/2\right)|^2\right)\left(\varphi_0-i\dfrac{\dt}{2}\left(-\dfrac{1}{2}\Delta+\beta|\varphi_0|^2\right)\varphi_0\right)\\
    & & +i\dt\beta\left(|\varphi\left(-\dt/2\right)|^2-\Upsilon_{-1/2}\right)\left(\varphi_0-i\dfrac{\dt}{2}\left(-\dfrac{1}{2}\Delta+\beta|\varphi_0|^2\right)\varphi_0\right) \\
    & & + i \dfrac{\dt}{2}\left( -\frac12\Delta+\beta\Upsilon_{1/2}\right)\int_0^{\dt}(\dt-\sigma)\partial_t^2\varphi(\sigma){\rm d}\sigma+ \int_0^{\dt} \dfrac{(\dt-\sigma)^2}{2}\partial_t^3\varphi(\sigma){\rm d}\sigma .
\end{eqnarray*}
Using $$\left|\varphi\left(-\dfrac{\dt}{2}\right)\right|^2=|\varphi(0)|^2+\dt\mathrm{Re}\left(i\overline{\varphi_0}\left(-\dfrac12\Delta\varphi_0+\beta|\varphi_0|^2\varphi_0\right)\right)+\int_0^{-\dt/2} (-\dt/2-\sigma)\partial_t^2(|\varphi|^2)(\sigma){\rm d}\sigma,$$
we obtain
\begin{eqnarray*}
\lefteqn{r(\dt)}\\\
    &=& i\dt\beta\left(|\varphi\left(-\dt/2\right)|^2-\Upsilon_{-1/2}\right)\left(\varphi_0-i\dfrac{\dt}{2}\left(-\dfrac{1}{2}\Delta+\beta|\varphi_0|^2\right)\varphi_0\right) \\
    & &+\dfrac{\dt^3}{2}\beta \mathrm{Re}\left(\dfrac{i}{2}\overline{\varphi_0}\Delta \varphi_0\right)\left(-\dfrac{1}{2}\Delta+\beta|\varphi_0|^2\right)\varphi_0 \\
    & & -i\dt\beta\left(\varphi_0-i\dfrac{\dt}{2}\left(-\dfrac{1}{2}\Delta+\beta|\varphi_0|^2\right)\varphi_0\right)\int_0^{-\dt/2} (-\dt/2-\sigma)\partial_t^2(|\varphi|^2)(\sigma){\rm d}\sigma \\
    & & + i \dfrac{\dt}{2}\left( -\frac12\Delta+\beta\Upsilon_{1/2}\right)\int_0^{\dt}(\dt-\sigma)\partial_t^2\varphi(\sigma){\rm d}\sigma+ \int_0^{\dt} \dfrac{(\dt-\sigma)^2}{2}\partial_t^3\varphi(\sigma){\rm d}\sigma.
      \end{eqnarray*}
  Note that we have
  \begin{equation*}
   \left\|\int_0^{\dt} \dfrac{(\dt-\sigma)^2}{2}\partial_t^3\varphi(\sigma){\rm d}\sigma\right\|_{H^{s+2}} \leq c\dt^3, \quad  \left\|
      i\frac{\dt}{2} \left(-\frac12\Delta+\beta\Upsilon_{1/2}\right)
          \int_0^{\dt}(\dt-\sigma)\partial_t^2\varphi(\sigma){\rm d}\sigma
    \right\|_{H^{s+2}} \leq c\dt^3,
  \end{equation*}
  \begin{equation*}
   \left\|i\dt\beta\left(\varphi_0-i\dfrac{\dt}{2}\left(-\dfrac{1}{2}\Delta+\beta|\varphi_0|^2\right)\varphi_0\right)\int_0^{-\dt/2} (-\dt/2-\sigma)\partial_t^2(|\varphi|^2)(\sigma){\rm d}\sigma\right\|_{H^{s+2}} \leq c\dt^3,   \end{equation*}
   and
    \begin{equation*}
   \left\|\dfrac{\dt^3}{2}\beta \mathrm{Re}\left(\dfrac{i}{2}\overline{\varphi_0}\Delta \varphi_0\right)\left(-\dfrac{1}{2}\Delta+\beta|\varphi_0|^2\right)\varphi_0\right\|_{H^{s+2}} \leq c\dt^3,   \end{equation*}
  where $c$ doesn't depend on $\dt$.
 Then we infer from the estimates above that 
  \begin{equation}
    \label{eq:estim5lem}
    \|r(\dt)\|_{H^{s+2}} \leq c \dt \left(\|\Upsilon_{-1/2}-\varphi(-\dt/2)\|_{H^{s+2}}
      +\dt^2
      \right).
    \end{equation}

    Now, we denote by $e(\dt)$ the fifth component $\varphi_1-\varphi(\dt)$
    of the vector $X_0-X(0)$ and we have
    \begin{equation*}
      e(\dt)=-i\dfrac{\dt}{2}\left(-\frac12\Delta+\beta\Upsilon_{1/2}\right) e(\dt)-r(\dt).
    \end{equation*}
    We want to estimate the $H^{s+2}$-norm of $e(\dt)$ using this relation
    and the estimate \eqref{eq:estim5lem}. To this aim, we take $\alpha\in\N^d$
    with $|\alpha|=\alpha_1+\dots+\alpha_d\leq s+2$ and differentiate
    the relation above to obtain that
    \begin{equation*}
      \partial_x^\alpha e(\dt)=i\dt\frac14\Delta\partial_x^\alpha e(\dt)
      -i\dfrac{\dt}{2}\beta\partial_x^\alpha (\Upsilon_{1/2} e(\dt))-\partial_x^\alpha r(\dt).
    \end{equation*}
    Multiplying this relation by $\overline{\partial_x^\alpha e(\dt)}$
    integrating over $\R^d$, and taking the real part, we obtain
    \begin{equation}
    \label{eq:estim6lem}
      \|\partial_x^\alpha e(\dt)\|_2^2=-\beta\dfrac{\dt}{2}\mathrm{Re}\left( i\int_{\R^d}
        \overline{\partial_x^\alpha e(\dt)} \partial_x^\alpha (\Upsilon_{1/2}e(\dt))
      \right) - \mathrm{Re}\left(\int_{\R^d}\overline{\partial_x^\alpha e(\dt)}\partial_x^\alpha r(\dt)\right).
    \end{equation}
    When $\alpha=0_{\N^d}$, the first term on the right hand side in the equation above vanishes and 
    $$\| e(\dt)\|_2^2 \leq  \| e(\dt)\|_2 \| r(\dt)\|_2,$$
    using Cauchy-Schwartz inequality. We infer, 
\begin{equation}
\label{eq:estim7lem}
\| e(\dt)\|_2 \leq  \| r(\dt)\|_2. 
\end{equation}
    Now, when $\alpha \neq 0_{\N^d}$, the Leibniz' rule applied in \eqref{eq:estim6lem} provides us with:
    \begin{equation*}
      \|\partial_x^\alpha e(\dt)\|_2^2=-\beta\dfrac{\dt}{2}\sum_{\{k:k\leq \alpha\}}c(k,\alpha)\mathrm{Re}\left( i\int_{\R^d}
        \overline{\partial_x^\alpha e(\dt)} \partial_x^k \Upsilon_{1/2}\partial_x^{\alpha-k}e(\dt)
     \right) - \mathrm{Re}\left(\int_{\R^d}\overline{\partial_x^\alpha e(\dt)}\partial_x^\alpha r(\dt)\right),
    \end{equation*}
    where $c(k,\alpha)$ are integers. Note that since $\Upsilon_{1/2}$ is real valued, the term corresponding to $k=0_{\N^d}$ in the sum vanishes. 
    Then, with Cauchy-Schwarz inequality, we infer
    \begin{eqnarray*}
      \|\partial_x^\alpha e(\dt)\|_2^2 &\leq& \beta\dfrac{\dt}{2}
                                               \sum_{\{k:k\leq \alpha, |k| \geq 1\}}c(k,\alpha) \|\partial_x^\alpha e(\dt)\|_2\|\partial_x^k \Upsilon_{1/2}\partial_x^{\alpha-k}e(\dt))\|_2
                                                + \|\partial_x^\alpha e(\dt)\|_2 \|\partial_x^\alpha r(\dt)\|_2.\\
     &\leq& \beta\dfrac{\dt}{2}
                                               \sum_{\{k:k\leq \alpha, |k| \geq 1\}}c(k,\alpha) \|\partial_x^\alpha e(\dt)\|_2\|\partial_x^k \Upsilon_{1/2}\|_{\infty}\|\partial_x^{\alpha-k}e(\dt))\|_2
                                                + \|\partial_x^\alpha e(\dt)\|_2 \|\partial_x^\alpha r(\dt)\|_2.                             
    \end{eqnarray*}
    Since $k \leq \alpha$ in the sum above, we have $|k|\leq |\alpha|$. For such $k$, one has
    $$\|\partial_x^k \Upsilon_{1/2}\|_\infty \leq c \|\partial_x^k \Upsilon_{1/2}\|_{H^{(d+1)/2}} \leq c \| \Upsilon_{1/2}\|_{H^{(d+1)/2+|k|}} \leq c \| \Upsilon_{1/2}\|_{H^{(d+1)/2+s+2}}  \leq c \| \Upsilon_{1/2}\|_{H^{s+4}},$$
    where $c$ is the Sobolev constant of the injection from $H^{(d+1)/2}(\R^d)$ to $L^\infty(\R^d)$ and where we have used the fact that $d \in \{1, 2, 3\}$. Since the $H^{s+4}$-norm of $\Upsilon_{1/2}$ is
    controlled by \eqref{eq:estim2eme}, we have
    \begin{equation*}
      \|\partial_x^\alpha e(\dt)\|_2 \leq c_\alpha
                                                \| e(\dt)\|_{H^{|\alpha|-1}}
                                                + \|\partial_x^\alpha r(\dt)\|_2,
    \end{equation*}
    where $c_\alpha$ does not depend on $\dt \in (0,\dt_0)$.
    Then, by induction on $|\alpha|\in\{0,\dots,s+2\}$ starting with \eqref{eq:estim7lem} for $\alpha=0_{\N^d}$, 
    we have  for some positive constant $c_{s+2}$:
    \begin{equation*}
      \|\partial_x^\alpha e(\dt)\|_2 \leq c_{s+2} \|r(\dt)\|_{H^{s+2}}, 
    \end{equation*}
    for all $\dt \in (0,\dt_0)$ and all $\alpha\in\N^d$
    with $|\alpha|\leq s+2$.
    Therefore, we have $\|e(\dt)\|_{H^{s+2}}$ is controlled by $\|r(\dt)\|_{H^{s+2}}$
    and the conclusion for the fifth term follows using \eqref{eq:estim5lem}:
\begin{equation}
\label{eq:estim8lem}
\|\varphi_1-\varphi(\dt)\|_{H^{s+2}} \leq c \dt \left(\|\Upsilon_{-1/2}-\varphi(-\dt/2)\|_{H^{s+4}}+\dt^2\right),
\end{equation} 
where $c$ does not depend on $\dt \in (0,\dt_0)$.

It remains to estimate the $H^{s+2}$-norm of the sixth and seventh components of the vector $X_0-X(0)$. We only give the details for the seventh term since the computation is similar and even simpler for the sixth term. Let us denote by $p(\dt)$ the consistency error defined as 
$$p(\dt)=\dfrac{\varphi(\dt)-\varphi_0}{\dt}-\partial_t \varphi\left(\dfrac{\dt}{2}\right). $$
Using a Taylor expansion, since the exact solution is a smooth function from $(-\dt_0,\dt_0)$ to $H^{s+2}$, we have 
\begin{equation}
\label{eq:estim9lem}
\|p(\dt)\|_{H^{s+2}} \leq c \dt^2.
\end{equation}
Let us denote by $q(\dt)=v_{1/2}-\partial_t \varphi \left(\dfrac{\dt}{2}\right)$ the seventh component of $X_0-X(0)$. We have 
$$q(\dt)-p(\dt)=\dfrac{\varphi_1 - \varphi(\dt)}{\dt}.$$
Using estimates \eqref{eq:estim8lem} and \eqref{eq:estim9lem} we have by triangle inequality, 
$$\|q(\dt)\|_{H^{s+2}} \leq \|p(\dt)\|_{H^{s+2}} + \dfrac{1}{\dt}\left\|\varphi_1 - \varphi(\dt)\right\|_{H^{s+2}}\leq c\dt^2 + c  \left(\|\Upsilon_{-1/2}-\varphi(-\dt/2)\|_{H^{s+4}}+\dt^2\right).$$
This concludes the proof of the lemma.
\end{proof}
We prove below that the relaxation method \eqref{eq:relaxclassic}
is of order 2.
\begin{theorem}\label{th:machinbidule}
  Assume $d$, $s$, $R_1>0$, $\varphi_{\text{in}}\in H^{s+4}(\R^d)$, $T<T^\star$ are given and
  satisfy Hypotheses \ref{hypobornes}.
There exists $\mathscr{C}>0$ and $\delta t_0>0$ (smaller than the one in the hypotheses)
such that for all $\Upsilon_{-1/2}\in H^{s+4}(\R^d)$ with
$\|\Upsilon_{-1/2}\|_{H^{s+4}(\R^d)}\leq R_1$, all
$n\in \mathbb{N}$ and all $\dt \in (0,\dt_0)$ with
$n\dt \leq T$,
  \begin{equation}
    \label{eq:1}
\| X_n-X(t_n)\|_{(H^s(\R^d))^7}     \leq \mathscr{C} \left ( \|\Upsilon_{-1/2}-|\varphi(-\dt/2)|^2\|_{H^{s+4}(\R^d)}+\dt^2 \right).
  \end{equation}
\end{theorem}

\begin{proof}
  First, the initial datum $X_0$ is computed from $\varphi_0=\varphi_{\text{in}}$ and
  $\Upsilon_{-1/2}$ using \eqref{eq:IDrelext}. Therefore, using Lemma \ref{lem:estiminit}, there
  exists a constant $c>0$ such that for all $\dt\in(0,\dt_0)$ and all $\|\Upsilon_{-1/2}\|_{H^{s+4}(\R^d)}\leq R_1$, we have the estimate \eqref{eq:estimDI}
\begin{equation*}
  \|X_0-X(0)\|_{(H^{s+2}(\R^d))^7} \leq
c \left(\|\Upsilon_{-1/2}-|\varphi(-\dt/2)|^2\|_{H^{s+4}(\R^d)} +\dt^2\right).
\end{equation*}
Second, we define the consistency error of the relaxation scheme at
time $t_n=n\dt\leq T$ by the formula
\begin{equation*}
  R_n(\dt) =\mathcal{B} X(t_n) + \dt \mathcal{C} \mathcal{M}(X(t_n),X(t_{n+1})) - X(t_{n+1}).
\end{equation*}
Substracting this definition from \eqref{eq:relextsimpl}, we obtain
\begin{equation}
\label{eq:errorrecursion}
  X_{n+1}-X(t_{n+1})=\mathcal{B} (X_n-X(t_n))+
 \dt\mathcal{C}\left(\mathcal{M}(X_n,X_{n+1})-\mathcal M(X(t_n),X(t_{n+1}))\right)
 +R_n(\dt).
\end{equation}
From now on, we set for all $n\in\N$ such that $n\dt\leq T$, $e_n=X_n-X(t_n)$.
Iterating the relation above, we obtain, as long as $n\dt\leq T$,
\begin{equation*}
  e_{n} = {\mathcal B}^n e_0 + \dt
  \sum_{k=0}^{n-1} \mathcal B^{n-k-1}
\mathcal{C} \left(\mathcal{M}(X_k,X_{k+1})-\mathcal M(X(t_k),X(t_{k+1}))\right)
  + \sum_{k=0}^{n-1} \mathcal B^{n-k-1} R_{k}(\dt).
\end{equation*}
This implies 
\begin{equation*}
  \begin{array}{rcl}
\displaystyle  \|e_{n}\|_{(H^s(\R^d))^7} &\leq & \displaystyle \|{\mathcal B}^n \mathcal{C}\mathcal{C}^{-1}e_0\|_{(H^s(\R^d))^7} \\
&&\displaystyle+ \dt
  \sum_{k=0}^{n-1} \|\mathcal B^{n-k-1}
\mathcal{C} \left(\mathcal{M}(X_k,X_{k+1})-\mathcal M(X(t_k),X(t_{k+1}))\right)\|_{(H^s(\R^d))^7}\\
&&\displaystyle
  + \sum_{k=0}^{n-1}\| \mathcal B^{n-k-1} \mathcal{C}\mathcal{C}^{-1}R_{k}(\dt)\|_{(H^s(\R^d))^7}\\
&\leq & \displaystyle b \|\mathcal{C}^{-1}e_0\|_{(H^s(\R^d))^7}\\
&&\displaystyle + \dt
  b\sum_{k=0}^{n-1} \|\left(\mathcal{M}(X_k,X_{k+1})-\mathcal M(X(t_k),X(t_{k+1}))\right)\|_{(H^s(\R^d))^7}\\
&&\displaystyle
  + b \sum_{k=0}^{n-1}\| \mathcal{C}^{-1}R_{k}(\dt)\|_{(H^s(\R^d))^7}\\
&\leq & \displaystyle b \|e_0\|_{(H^{s+2}(\R^d))^7}\\
&&\displaystyle + 
C\dt b\sum_{k=0}^{n-1} \left ( \|e_k\|_{(H^s(\R^d))^7}+\|e_{k+1}\|_{(H^s(\R^d))^7}\right )
\\
&&\displaystyle
  + b \sum_{k=0}^{n-1}\| R_{k}(\dt)\|_{(H^{s+2}(\R^d))^7},\\
  \end{array}
\end{equation*}
using Lemmas \ref{lem:method} and \ref{lem:puissancebloc}. 
Then 
\begin{equation*}
\displaystyle{  (1-C\delta t b)\|e_{n}\|_{(H^s(\R^d))^7} \leq  b \|e_0\|_{(H^{s+2}(\R^d))^7} + 2
C\dt b\sum_{k=0}^{n-1} \|e_k\|_{(H^s(\R^d))^7}  + b \sum_{k=0}^{n-1}\| R_{k}(\dt)\|_{(H^{s+2}(\R^d))^7}.}
\end{equation*}
Let $\delta t$ be small enough to ensure that
$1-C\delta t b \geq \dfrac{1}{2}$. This implies 
\begin{equation*}
\displaystyle{  \|e_{n}\|_{(H^s(\R^d))^7} \leq  2b \|e_0\|_{(H^{s+2}(\R^d))^7} + 4
C\dt b\sum_{k=0}^{n-1} \|e_k\|_{(H^s(\R^d))^7}  + 2b \sum_{k=0}^{n-1}\| R_{k}(\dt)\|_{(H^{s+2}(\R^d))^7}.}
\end{equation*}
Taylor expansions and the fact that one can differentiate the last two lines of \eqref{eq:equivrelax} with respect to time show that there exists a constant $Q$ such that for all $n$ and $\dt$  with $n \dt \leq T$
\begin{equation*}
  \| R_{k}(\dt)\|_{(H^{s+2}(\R^d))^7}\leq Q \dt^3.
\end{equation*}
This implies that
$2b  \sum_{k=0}^{n-1}\| R_{k}(\dt)\|_{(H^{s+2}(\R^d))^7} \leq 2bQ T \dt^2$.
Then we obtain
\begin{equation*}
\displaystyle{  \|e_{n}\|_{(H^s(\R^d))^7} \leq  2b \|e_0\|_{(H^{s+2}(\R^d))^7}+ 2bQ T \dt^2+ 4
C\dt b\sum_{k=0}^{n-1} \|e_k\|_{(H^s(\R^d))^7}.}
\end{equation*}
Using a discrete Gronwall Lemma
(see section 5 in \cite{Holte09}), we get
\begin{equation*}
\displaystyle{  \|e_{n}\|_{(H^s(\R^d))^7}
\leq  (2b \|e_0\|_{(H^{s+2}(\R^d))^7}+ 2bQ T \dt^2)\exp(4CT b)}.
\end{equation*}
This estimate and \eqref{eq:estimDI} prove the result.
\end{proof}

\section{The generalized relaxation method for general nonlinearities}
\label{sec:genrelax}

\subsection{Generalized relaxation method for NLS equation}
We start by considering the simplified equation \eqref{eq:GPE}
with zero potential ($V=0$), no convolution operator ($U=0$) and
without rotation ($\Omega=0$). Assuming $\sigma\in\N^\star$,
we are therefore dealing with the classical nonlinear Schr\"odinger equation
\begin{equation}
  \label{eq:NLS}
  i\partial_t\varphi(t,x) = -\frac{1}{2} \Delta \varphi(t,x)
+ \beta |\varphi(t,x)|^{2\sigma} \varphi(t,x),
\end{equation}
with initial datum $\varphi_{\text{in}}$.

The original relaxation method applied to \eqref{eq:NLS} would
consists in adding the variable $\Upsilon$ to \eqref{eq:NLS} and
discretizing the following continuous system as discrete times $t_n$
and $t_{n+1/2}$
  \begin{equation}
  \left \{
\begin{array}{l}
  \Upsilon(t,x) = |\varphi(t,x)|^{2\sigma},\\
\displaystyle   i\partial_t\varphi(t,x) = -\frac{1}{2} \Delta \varphi(t,x)+ \beta \Upsilon \varphi(t,x).
\end{array}\right .\label{eq:false_relax}
\end{equation}
It is however known to not conserve energy functional. 

\noindent
As a generalization of the classical relaxation method \eqref{eq:relaxclassic},
which is designed for the special cubic case ($\sigma=1$),
we propose the following method, which allows for
a general nonlinearity exponent $\sigma\in\N^\star$. We propose to
substitute system \eqref{eq:false_relax} by
\begin{equation}
  \label{eq:cont_relax_gen}
\left \{
\begin{array}{l}
  {\gamma^\sigma(t,x)} = |\varphi(t,x)|^{2\sigma},\\
\displaystyle   i\partial_t\varphi(t,x) = -\frac{1}{2} \Delta \varphi(t,x)+ \beta {\gamma^\sigma} \varphi(t,x).
\end{array}\right .  
\end{equation}
The modification seems ligth but allows to build an energy preserving
scheme. The second equation is approximated to second order at time
$t_{n+1/2}$ by
\[
      \displaystyle i\frac{\varphi_{n+1}-\varphi_n}{\dt}
 = \left( -\frac{1}{2} \Delta 
+ \beta \gamma_{n+1/2}^\sigma
\right)\frac{\varphi_{n+1}+\varphi_n}{2}.
\]
We now want to find an approximation of $\gamma^\sigma =
|\varphi|^{2\sigma}=\gamma^{\sigma-1}|\varphi|^2$ that allow energy
conservation following the ideas that were presented in section
\ref{sec:pres}. As for the classical relaxation method, we note that
\[\mathrm{Re}\left( \gamma_{n+1/2}^\sigma
  \frac{\varphi_{n+1}+\varphi_n}{2}\overline{\varphi_{n+1}-\varphi_n}\right)=\frac{\gamma_{n+1/2}^\sigma}{2}\Big(|\varphi_{n+1}|^2-|\varphi_{n}|^2\Big).\]
The last term also reads
\[
\begin{array}{ll}
  \displaystyle \gamma_{n+1/2}^\sigma\Big(|\varphi_{n+1}|^2-|\varphi_{n}|^2\Big)& \displaystyle =\gamma_{n+1/2}^\sigma\Big(|\varphi_{n+1}|^2-|\varphi_{n}|^2\Big)+\gamma_{n-1/2}^\sigma\Big(|\varphi_{n}|^2-|\varphi_{n}|^2\Big) \\
  & = \displaystyle \Big(\gamma^\sigma_{n+1/2}|\varphi_{n+1}|^2-\gamma^\sigma_{n-1/2}|\varphi_{n}|^2\Big)-{\Big(\gamma^\sigma_{n+1/2} -\gamma^\sigma_{n-1/2}\Big)|\varphi_{n}|^2}.
\end{array}
\]
The only choice that allow to preserve energy is to choose
\[
\Big(\gamma^\sigma_{n+1/2} -\gamma^\sigma_{n-1/2}\Big)|\varphi_{n}|^2
= \frac{\sigma}{\sigma+1} \Big (
\gamma_{n+1/2}^{\sigma+1}-\gamma_{n-1/2}^{\sigma+1}\Big ).\]
Moreover, we remark that
\[
\frac{1}{\sigma+1} \frac{
  \gamma_{n+1/2}^{\sigma+1}-\gamma_{n-1/2}^{\sigma+1}}{\delta
  t}=\frac{1}{\sigma} \frac{\gamma^\sigma_{n+1/2}
  -\gamma^\sigma_{n-1/2}}{\delta t}|\varphi_{n}|^2
\]
is a second order approximation of
$\gamma^\sigma=\gamma^{\sigma-1}|\varphi|^2$ at time $t=t_n$.

This method is therefore designed so that it preserves exactly the following energy
\begin{equation*}
  E_{\mathrm{rlx}}(\varphi,\gamma) = \frac{1}{4} \int_{\R^d} \|\nabla \varphi\|^2 \dd x
+ \frac{\beta}{2} \int{\R^d} \gamma^\sigma |\varphi|^2 \dd x
- \beta \frac{\sigma}{2(\sigma+1)} \int{\R^d} \gamma^{\sigma+1} \dd x.
\end{equation*}
Since at continuous level $\gamma^\sigma = |\varphi|^{2\sigma}$,
$E_{\mathrm{rlx}}(\varphi,\gamma)$ reduces to the true energy \[\frac{1}{4} \int_{\R^d} \|\nabla \varphi\|^2 \, \dd x
+ \frac{\beta}{2\sigma+2} \int_{\R^d} |\varphi|^{2\sigma+2}\, \dd x.\]
Moreover, the generalized relaxation method preserves the $L^2$-norm
of the solution. We present in Theorem \ref{theoremNRJpreservation} a
more general method, which includes possibly
non zero other terms in the equation (see Section
\ref{subsec:relaxgen}).\\

Starting with $\varphi_0=\varphi_{\text{in}}$ and $\gamma_{-1/2}$ approximating
$|\varphi(-\dt/2)|^2$, the generalized relaxation method is
\begin{equation}
  \label{eq:relaxgena}
    \left\{
    \begin{array}{l}
      \displaystyle
\frac{\gamma_{n+1/2}^{\sigma+1}-\gamma_{n-1/2}^{\sigma+1}}{\gamma_{n+1/2}^{\sigma}-\gamma_{n-1/2}^{\sigma}}=\frac{\sigma + 1}{\sigma} |\varphi_n|^2,\\[5 mm]
      \displaystyle i\frac{\varphi_{n+1}-\varphi_n}{h}
 = \left( -\frac{1}{2} \Delta 
+ \beta \gamma_{n+1/2} ^\sigma
\right)\frac{\varphi_{n+1}+\varphi_n}{2}.
    \end{array}
  \right.
\end{equation}
Like for the Crank-Nicolson scheme and equation \eqref{eq:jesaispas}, the first equation also reads
\[
      \gamma_{n+1/2}^\sigma = \left(
\frac{\sigma+1}{\sigma} |\varphi_n|^2 - \gamma_{n-1/2}
\right)
\left(
\sum_{k=0}^{\sigma-1} \gamma_{n+1/2}^k \gamma_{n-1/2}^{\sigma-1-k}
\right).
\]
so we also have the other version of generalized relaxation method
\begin{equation}
\label{eq:relaxgen}
  \left\{
    \begin{array}{l}
      \displaystyle \gamma_{n+1/2}^\sigma = \left(
\frac{\sigma+1}{\sigma} |\varphi_n|^2 - \gamma_{n-1/2}
\right)
\left(
\sum_{k=0}^{\sigma-1} \gamma_{n+1/2}^k \gamma_{n-1/2}^{\sigma-1-k}
\right),\\[5 mm]
      \displaystyle i\frac{\varphi_{n+1}-\varphi_n}{\dt}
 = \left( -\frac{1}{2} \Delta 
+ \beta \gamma_{n+1/2} ^\sigma
\right)\frac{\varphi_{n+1}+\varphi_n}{2}.
    \end{array}
  \right.
\end{equation}
  
\noindent
Note that, when $\sigma=1$, the generalized relaxation method
\eqref{eq:relaxgen} reduces to the classical relaxation method
\eqref{eq:relaxclassic}.
In contrast to the classical relaxation method \eqref{eq:relaxclassic},
when $\sigma \geq 2$,
the generalized relaxation method \eqref{eq:relaxgen} is
fully implicit on its first stage, and linearly implicit in its second
stage. For small values of $\sigma$ ($\sigma=1, 2, 3 ,4$) the first
stage of \eqref{eq:relaxgen} is polynomial of degree $\sigma$ and
hence we can use explicit formulas for the computation of the solution
$\gamma_{n+1/2}$.
For example, for quintic nonlinearity and $\sigma=2$, we have explicit
solutions  to the quadratic equation
\[\gamma_{n+1/2}^2-k_1\gamma_{n+1/2}-k_1\gamma_{n-1/2}=0,\]
where $k_1=(3/2)|\varphi_n|^2-\gamma_{n-1/2}$.

For higher values of $\sigma$, one can use the
following iterative fixed-point procedure, 
starting with $\gamma_{n+1/2,0}=|\varphi_n|^2$ :
\begin{equation}
\label{eq:relaxiter}
\forall p\in\N,\qquad
\gamma_{n+1/2,{p+1}} = \left(
\frac{\sigma+1}{\sigma} |\varphi_n|^2 - \gamma_{n-1/2}
\right)^{1/\sigma}
\left(
\sum_{k=0}^{\sigma-1} \gamma_{n+1/2,p}^k \gamma_{n-1/2}^{\sigma-1-k}
\right)^{1/\sigma},  
\end{equation}
which one stops when $\|\gamma_{n+1/2,p+1}-\gamma_{n+1/2,p}\|_{L^2}$ is below
some small tolerance parameter, and for this index $p$, one sets
$\gamma_{n+1/2}=\gamma_{n+1/2,p+1}$.

Numerically, this generalized relaxation method has order 2,
as we will see in the numerical experiments section \ref{sec:num}.
However, we do not address this theoretical question in this paper.
In the next subsection, we show how one can design a generalized relaxation
method similar to \eqref{eq:relaxgen} in order to treat the cases
with non-zero potential $V$, non-zero convolution operator $U$
or non-zero rotation $\Omega$.

\subsection{A generalized relaxation method for GPE}
\label{subsec:relaxgen}

We propose the following generalization of the relaxation method introduced
in \cite{Besse2004}, which uses two additional unknowns $\gamma$ and $\Upsilon$:
starting from
$\varphi_0=\varphi_{\text{in}}$, we initialize $\Upsilon_{-1/2}$ and $\gamma_{-1/2}$
with approximations of $|\varphi(-\dt/2)|^2$
and compute for $n\in\N$,
$(\varphi_{n+1},\Upsilon_{n+1/2},\gamma_{n+1/2})$
from $(\varphi_n,\Upsilon_{n-1/2},\gamma_{n-1/2})$
using the formulae

\begin{equation}
\label{Frelaxation}
  \left\{
    \begin{array}{l}
      \displaystyle \gamma_{n+1/2}^\sigma = \left(
\frac{\sigma+1}{\sigma} |\varphi_n|^2 - \gamma_{n-1/2}
\right)
\left(
\sum_{k=0}^{\sigma-1} \gamma_{n+1/2}^k \gamma_{n-1/2}^{\sigma-1-k}
\right),\\[5 mm]
      \displaystyle\frac{\Upsilon_{n+1/2}+\Upsilon_{n-1/2}}{2} = |\varphi_n|^2,\\
      \displaystyle i\frac{\varphi_{n+1}-\varphi_n}{\dt}
 = \left( -\frac{1}{2} \Delta 
+ V 
+ \beta \gamma_{n+1/2}^\sigma 
+ \lambda (U\ast\Upsilon_{n+1/2}) 
- \Omega.R
\right)\frac{\varphi_{n+1}+\varphi_n}{2}.
    \end{array}
  \right.
\end{equation}

The first equation of \eqref{Frelaxation} is implicit and local.
In order to solve it, we propose two different approaches.
The first one, for small integer values of $\sigma$, consists in using
exact formulae, since the equation is polynomial of low degree in these cases.
The second approach, is to use a
fixed point method, as described in \eqref{eq:relaxiter}.
The second equation is explicit, as in the classical relaxation method
\eqref{eq:relaxclassic}.
The third equation of \eqref{Frelaxation} is implicit and non-local.
We propose again two different approaches to solve it.
The first approach, presented here for $x\in \T_\delta^d$, consists in using another fixed point iteration method:
one starts from $\varphi_{n+1}^0=\varphi_{n}$
and computes $\varphi_{n+1}^{p+1}$ from $\varphi_{n+1}^p$ using the
iterative procedure
\begin{eqnarray}
\label{eq:rlxpointfixeFourier}
\lefteqn{
  \left(1+i\frac{\dt}{2}\frac{\xi^2}{2}\right) {\mathcal F}(\varphi_{n+1}^{p+1})(\xi)} & \\ \nonumber
  =& \displaystyle
  \left(1-i\frac{\dt}{2}\frac{\xi^2}{2}\right) {\mathcal F}(\varphi_{n})(\xi)
  -i\dt{\mathcal F}\left(
   \left(V +\beta \gamma_{n+1/2}^\sigma  
+ \lambda (U\ast\Upsilon_{n+1/2}) 
- \Omega.R
\right)\left(\frac{\varphi_{n+1}^p+\varphi_n}{2}\right)\right)(\xi),
\end{eqnarray}
where $\mathcal F$ stands for the Fourier transform in space as defined
in Appendix \ref{sec:notation},
and we set $\varphi_{n+1}=\lim_{p\to \infty} \varphi_{n+1}^p$.
In practice, the iterative procedure
$\varphi_{n+1}^p\rightarrow \varphi_{n+1}^{p+1}$ stops
whenever the $L^2$-norm of the difference
between two consecutive steps is below some small tolerance parameter.
Note that we decided to implicit the Laplace operator in the
iterative procedure \eqref{eq:rlxpointfixeFourier} in order
to ensure that the Sobolev regularity of $\varphi_{n+1}^p$ does
no decrease {\it a priori} when $p$ increases.
Alternatively, the other approach to solve the last equation of \eqref{Frelaxation} consists in following \cite{XavRom2014},
and using the linearity of the equation.
Let us introduce the new unknown
$\varphi_{n+1/2}=(\varphi_{n+1}+\varphi_n)/2$, so that the equation reads
\begin{equation*}
\left (I - i\frac{\delta t}{2}\frac{\Delta}{2} + i\frac{\delta t}{2}
  V + i \frac{\delta t}{2} \beta
  \gamma_{n+1/2}^\sigma+i \frac{\delta t}{2} \lambda(U\ast
  \Upsilon_{n+1/2})- i\frac{\delta t}{2} \Omega \cdot
  R\right) \varphi_{n+1/2}=\varphi_n.
\end{equation*}
This equation can also be written as
\begin{equation}
  \label{eq:precondkryl}
\mathcal{L}\varphi_{n+1/2}=P^{-1}\varphi_n,  
\end{equation}
where
\begin{equation*}
\mathcal{L}=\left (I + i\frac{\delta t}{2}P^{-1}\left (  V + \beta
  \gamma_{n+1/2}^\sigma+ \lambda(U\ast
  \Upsilon_{n+1/2})- \Omega \cdot
  R\right )\right) ,
\end{equation*}
and
\begin{equation*}
P=\left(I -i\frac{\delta t}{2}\frac{\Delta}{2}\right).
\end{equation*}
The precondionning operator $P$ can be easily inverted in Fourier space
and the solution of \eqref{eq:precondkryl} can be obtained by a Krylov
method. Note that other choices of preconditionning operator are
possible (see \cite{XavRom2014}).

In the following, we shall use the notation
\begin{equation*}
  (\varphi_{n+1},\gamma_{n+1/2},\Upsilon_{n+1/2})
= \Phi_{\dt}^{\rm rlx} (\varphi_n,\gamma_{n-1/2},\Upsilon_{n-1/2}),
\end{equation*}
for the method \eqref{Frelaxation} above. 

\subsubsection{Energy preservation property for generalized relaxation methods} 
The generalized relaxation method \eqref{Frelaxation}
is designed to preserve exactly the following energy:
\begin{eqnarray}
\label{eq:NRJrlx}
  E_{\rm rlx}(\varphi,\gamma,\Upsilon) & = & \int_{\R^d} \left( \frac{1}{4} 
\|\nabla \varphi\|^2 
 + \frac{1}{2} V |\varphi|^2 
 + \frac{\beta}{2} \gamma^\sigma |\varphi|^2
 -\beta \frac{\sigma}{2(\sigma+1)} \gamma^{\sigma+1} \right) \dd x \\ \nonumber
 & + &\int_{\R^d} \left(\frac{\lambda}{2} (U \ast \Upsilon) |\varphi|^2
  - \frac{\lambda}{4} (U \ast \Upsilon) \Upsilon
  - \frac{\Omega}{2} \overline{\varphi} R \varphi \right)\dd x,
\end{eqnarray}
as we prove in Theorem \ref{theoremNRJpreservation}.
Note that this energy is consistent with the
energy \eqref{eq:NRJGPE} of equation \eqref{eq:GPE} in the sense that
\begin{equation*}
  E_{\rm rlx}(\varphi,|\varphi|^2,|\varphi|^2) = E(\varphi).
\end{equation*}

\begin{theorem}
\label{theoremNRJpreservation}
  The generalized relaxation method \eqref{Frelaxation} applied
to the equation \eqref{eq:GPE} with initial data $\varphi_0=\varphi_{\text{in}}$,
$\gamma_{-1/2}$ and $\Upsilon_{-1/2}$ preserves
exactly the $L^2$ norm and the energy functional $E_{\rm rlx}$ defined in
\eqref{eq:NRJrlx} in the sense that for all $n\in\N$ such that
a solution of \eqref{Frelaxation} is defined, we have
\begin{equation}
\label{eq:proptheoremNRJpreservation}
  \|\varphi_{n+1}\|^2 = \|\varphi_0\|^2 \qquad \text{and} \qquad
  E_{\rm rlx} (\varphi_{n+1},\gamma_{n+1/2},\Upsilon_{n+1/2})
= E_{\rm rlx} (\varphi_0,\gamma_{-1/2},\Upsilon_{-1/2}).
\end{equation}
\end{theorem}

\begin{proof}
  Multiplying the last equation of \eqref{Frelaxation} by
$\overline{\varphi_{n+1}-\varphi_n}$, integrating in space, and taking the real
part, one finds that a sum of 5 terms is equal to
\begin{equation}
\label{eq:presNRJ0}
 \int_{\R^d}\mathrm{Re} \left(
i |\varphi_{n+1}-\varphi_n|^2 
\right) dx=0.
\end{equation}
The first term reads
\begin{equation}
  \label{eq:rlxNRJcin}
\int_{\R^d}\mathrm{Re}\left(
-\frac{1}{2} \big(\Delta \frac{\varphi_{n+1}+\varphi_n}{2}\big)
\big(\overline{\varphi_{n+1}-\varphi_n} \big)
\right) dx= \frac{1}{4} \int_{\R^d} \|\nabla\varphi_{n+1}\|^2 dx - \frac{1}{4}
\int_{\R^d} \|\nabla \varphi_n\|^2 dx.
\end{equation}
The second term reads
\begin{equation}
  \label{eq:rlxNRJpot}
  \int_{\R^d}\mathrm{Re}\left(
 \big(V \frac{\varphi_{n+1}+\varphi_n}{2}\big)
\big(\overline{\varphi_{n+1}-\varphi_n} \big)
\right) dx= \frac{1}{2} \int_{\R^d} V |\varphi_{n+1}|^2 dx- \frac{1}{2}
\int_{\R^d} V |\varphi_n|^2dx.
\end{equation}
Using the first equation of \eqref{Frelaxation}, the third term reads
\begin{eqnarray}
  \label{eq:rlxNRJNL}
 \int_{\R^d} \mathrm{Re}\left(
\beta \gamma_{n+1/2}^\sigma \big( \frac{\varphi_{n+1}+\varphi_n}{2}\big)
\big(\overline{\varphi_{n+1}-\varphi_n} \big)
\right) dx & = &
\beta\left(\frac{1}{2} \int_{\R^d} \gamma_{n+1/2}^\sigma |\varphi_{n+1}|^2 dx
- \frac{\sigma}{2(\sigma+1)}
\int_{\R^d} \gamma_{n+1/2}^{\sigma+1} dx \right)  \nonumber\\ 
& - &
\beta\left(\frac{1}{2} \int_{\R^d} \gamma_{n-1/2}^\sigma |\varphi_{n}|^2  dx
- \frac{\sigma}{2(\sigma+1)}
\int_{\R^d}\gamma_{n-1/2}^{\sigma+1}dx \right)
.
\end{eqnarray}
Using the fact that the convolution kernel $U$ is real-valued and symmetric
with respect to the origin, and also using the second equation of \eqref{Frelaxation}, the fourth term reads
\begin{eqnarray}
\label{eq:rlxnonlocalNRJ}
\lefteqn{
  \lambda \int_{\R^d} \mathrm{Re}\Big( (U\ast\Upsilon_{n+1/2}) \frac{\varphi_n+\varphi_{n+1}}{2}
\overline{\varphi_{n+1}-\varphi_n}
\Big) dx} & & \\
& = & \lambda \Big(
\frac{1}{2} \int_{\R^d} (U\ast \Upsilon_{n+1/2}) |\varphi_{n+1}|^2 dx
- \frac{1}{4} \int_{\R^d} (U\ast \Upsilon_{n+1/2}) \Upsilon_{n+1/2} dx
\Big)\nonumber\\
& - & \lambda \Big(
\frac{1}{2} \int_{\R^d} (U\ast \Upsilon_{n-1/2}) |\varphi_{n}|^2 dx
- \frac{1}{4} \int_{\R^d} (U\ast \Upsilon_{n-1/2}) \Upsilon_{n-1/2} dx
\Big).\nonumber
\end{eqnarray}
Finally, thanks to the fact that the operator $R$ is symmetric, 
the last term reads
\begin{equation}
\label{eq:rlxNRJrot}
  -\Omega \int_{\R^d} \mathrm{Re} \left(
\Big(R\frac{\varphi_{n+1}+\varphi_n}{2}\Big) \overline{\varphi_{n+1}-\varphi_n}
\right)=
-\Omega\int_{\R^d} (\varphi_{n+1} R \overline{\varphi_{n+1}})
+ \Omega\int_{\R^d} (\varphi_{n} R \overline{\varphi_{n}}).
\end{equation}
The proof of the preservation of the energy
\eqref{eq:proptheoremNRJpreservation} is completed
by adding altogether \eqref{eq:rlxNRJcin}, \eqref{eq:rlxNRJpot},
\eqref{eq:rlxNRJNL}, \eqref{eq:rlxNRJNL}, and \eqref{eq:rlxNRJrot},
and using relation \eqref{eq:presNRJ0}.
The preservation of the $L^2$-norm follows from multiplying the
last relation in \eqref{Frelaxation} by $\overline{\varphi_{n+1}+\varphi_n}$,
integrating in space, and taking the imaginary part.
\end{proof}

\section{Numerical experiments}
\label{sec:num}
In this section, we make some numerical experiments and show that the
classical and generalized relaxation methods are efficient methods
that preserve mass and energy to machine epsilon.

\subsection{One dimensional example: the one-dimensional quintic and septic NLS equation}

We present in this subsection some numerical experiments to show the efficiency
of the generalized relaxation \eqref{eq:relaxgen} compared with the
Crank-Nicolson scheme \eqref{eq:CNmethod}
when considering
\begin{equation}
  \label{eq:gennls}
  i\partial_t\varphi(t,x) = -\frac{1}{2} \Delta \varphi(t,x)
+ \beta |\varphi(t,x)|^{2\sigma} \varphi(t,x), \quad \sigma=2,3,
\end{equation}
with $\varphi(0,x)=\varphi_{\mathrm{in}}(x)$ and $(t,x)\in
[0,T]\times (x_\ell,x_r)$. To deal with the space variable, we discretize
the space operators using Fourier  spectral approximation and consider
periodic boundary conditions. The spatial mesh size is
$\delta x >0$ with $\delta x=(x_r-x_\ell)/J$ with $J=2^P$, $P\in \N^*$. The
time step is $\delta t=T/N$ for some
$N\in\N^\star$. The grid points and the discrete times are 
\[
x_j:=x_\ell+j\delta x, \quad t_n:=n\delta t, \quad j=0,1,\cdots,J, \quad n=0,1,\cdots,N.
\]
Let $\varphi_{j,n}$ be the approximation of $\varphi(t_n,x_j)$ satisfying
\[
\varphi_{j,n} = \frac{1}{J} \sum_{k=0}^{J-1} \hat{\varphi}_{k,n} \omega_J^{jk}, \quad j=0,\cdots,J-1,
\]
where $\hat{\varphi}_{k,n}$ denotes the discrete Fourier transform of
sequence $(\varphi_{j,n})_j$ given by
\[
\hat{\varphi}_k^n=\sum_{q=0}^{J-1} \varphi_q^n \omega_J^{-jk}, \quad k=-\frac{J}{2},\cdots,\frac{J}{2}-1,
\]
where $\omega_J=\exp{(-2i\pi/J)}$. We also apply the discrete Fourier
transform to the approximation $\gamma_{j,{n+1/2}}$ of $|\varphi(t_{n+1/2},x_j)|^2$.
Let us define the discrete gradient operator $\nabla_d$
\[
\widehat{(\nabla_d v)}_k = i \mu_k  \hat{v}_k, \quad v\in \C^J.
\] 
Let us denote by $\Pi_d$ the projection operator
\[
\begin{array}{llcl}
\Pi_d: & \mathcal{C}^0([x_\ell,x_r],\C) & \to & \C^{J}\\
       & \varphi & \mapsto & \left (\varphi(x_j) \right )_{0\leq j \leq J-1}  
\end{array}.
\]
We define the discrete $\ell^r$ norm on $\C^J$ as
\[
\|v\|_{\ell^r} = \left (\delta x \sum_{j=0}^{J-1} |v_j|^r\right)^{1/r}, \quad v\in \C^J, \ 
r\geq 1,
\]
the mean
\[
M(v) = \delta x \sum_{j=0}^{J-1} v_j, \quad v\in \C^J, \ 
r\geq 1,
\]
and the discrete energies:
\[
E_{d}(v)=\frac{1}{4}\left \|\nabla_d v \right
\|_{\ell^2}^2+\frac{\beta}{2(\sigma+1)} { \left
    \||v|^{\sigma+1}\right \|_{\ell^2}^2},
\]
and
\[
E_{\mathrm{rlx},d}(v,g)=\frac{1}{4}\left \|\nabla_d v \right
\|_{\ell^2}^2+\frac{\beta}{2} M \left(
    g^\sigma(|v|^2-\frac{\sigma}{\sigma+1}g)\right).
\]
As in the continuous case, we have $E_d(v)=E_{\mathrm{rlx},d}(v,v)$
for any $v\in\C^J$. 

Using these definitions, the energy conservation for the relaxation scheme is built through the following relative error
\begin{equation}
  \label{eq:errener}
  \mathcal{E}_{E,\delta t}=\sup_{n\in\{0,\cdots,N\}} \frac{\left
    |E_{\mathrm{rlx},d}(\Pi_d(\varphi_{ex}(0,\cdot)),\Pi_d(\varphi_{ex}(-\delta
    t/2,\cdot)))-E_{\mathrm{rlx},d}((\varphi_j^n)_j,(\gamma_j^{n-1/2})_j)\right |}{E_{\mathrm{rlx},d}(\Pi_d(\varphi_{ex}(0,\cdot)),\Pi_d(\varphi_{ex}(-\delta
    t/2,\cdot)))}.
\end{equation}
For the Crank-Nicolson scheme, it is defined by
\begin{equation}
  \label{eq:errenercn}
  \mathcal{E}_{E,\delta t}=\sup_{n\in\{0,\cdots,N\}} \frac{\left
    |E_{d}(\Pi_d(\varphi_{ex}(0,\cdot)))-E_{d}((\varphi_j^n)_j)\right |}{E_{d}(\Pi_d(\varphi_{ex}(0,\cdot)))}.
\end{equation}

We present on Figure \ref{fig:un} the evolution of $\mathcal{E}_{E,\delta t}$ for various $\delta t$ when $\sigma=2$ and $\sigma=3$ for both classical and generalized relaxation scheme compared to Crank-Nicolson scheme. The initial datum is chosen to be $\varphi_{\mathrm{in}}(x)=\exp(-x^2)$ and $\beta=-1$. The time-space domain is $[0,1/2]\times[-30,30]$. We consider periodic boundary conditions and the interval $[-30,30]$ is meshed with $2^{13}+1$ nodes. As it is known, the standard relaxation scheme does not preserve the energy when $\sigma=2$ or $\sigma=3$ and the error curve show a second order convergence. On the other hand, both generalized relaxation and Crank-Nicolson preserve energy to epsilon machine.

\newlength\figureheight
\newlength\figurewidth
\setlength\figureheight{4cm}
\setlength\figurewidth{6cm}
\begin{figure}[htbp]
  \centering
  \begin{tabular}{cl}
\begin{minipage}{0.08\textwidth}$\sigma=2$,\end{minipage} & \begin{minipage}{0.7\textwidth}
%
%
\definecolor{mycolor1}{rgb}{0.00000,0.44700,0.74100}%
\definecolor{mycolor2}{rgb}{0.85000,0.32500,0.09800}%
\definecolor{mycolor3}{rgb}{0.92900,0.69400,0.12500}%
\begin{tikzpicture}

\begin{axis}[%
width=0.951\figurewidth,
height=\figureheight,
at={(0\figurewidth,0\figureheight)},
scale only axis,
xmode=log,
xmin=0.0001,
xmax=0.1,
xminorticks=true,
xlabel style={font=\color{white!15!black}},
xlabel={$\delta t$},
ymode=log,
ymin=1e-16,
ymax=0.01,
yminorticks=true,
ylabel style={font=\color{white!15!black}},
ylabel={$\mathcal{E}_{E,\delta t}$},
axis background/.style={fill=white},
legend style={at={(1.03,1)}, anchor=north west, legend cell align=left, align=left, draw=white!15!black}
]
\addplot [color=mycolor1, mark=*, mark options={solid, fill=mycolor1, mycolor1}]
  table[row sep=crcr]{%
0.0325714598997343	0.000142853291310846\\
0.0103	1.33891039657853e-05\\
0.00325714598997343	1.30793711552361e-06\\
0.00103	1.29823410724186e-07\\
0.000325714598997343	1.29517647847715e-08\\
0.000103	1.2943261910617e-09\\
};
\addlegendentry{Classical relaxation}

\addplot [color=mycolor2, mark=square*, mark options={solid, fill=mycolor2, mycolor2}]
  table[row sep=crcr]{%
0.0325714598997343	2.592250902952e-15\\
0.0103	7.48872758882481e-15\\
0.00325714598997343	5.76055970611882e-15\\
0.00103	8.64083955951626e-15\\
0.000325714598997343	9.44731791840814e-14\\
0.000103	1.55679126064013e-13\\
};
\addlegendentry{Generalized relaxation}

\addplot [color=mycolor3, mark=diamond*, mark options={solid, fill=mycolor3, mycolor3}]
  table[row sep=crcr]{%
0.0325714598997343	3.16830783849055e-15\\
0.0103	3.31232183114921e-15\\
0.00325714598997343	4.75246175773582e-15\\
0.00103	4.4644337724185e-15\\
0.000325714598997343	5.14129953791421e-14\\
0.000103	8.30960737640476e-14\\
};
\addlegendentry{Crank Nicolson}

\addplot [color=black]
  table[row sep=crcr]{%
0.0325714598997343	1e-06\\
0.0103	1e-07\\
};
\addlegendentry{Slope 2}

\end{axis}
\end{tikzpicture}%
\end{minipage}\\
\begin{minipage}{0.08\textwidth}$\sigma=3$,\end{minipage} & \begin{minipage}{0.7\textwidth}
%
%
\definecolor{mycolor1}{rgb}{0.00000,0.44700,0.74100}%
\definecolor{mycolor2}{rgb}{0.85000,0.32500,0.09800}%
\definecolor{mycolor3}{rgb}{0.92900,0.69400,0.12500}%
\begin{tikzpicture}

\begin{axis}[%
width=0.951\figurewidth,
height=\figureheight,
at={(0\figurewidth,0\figureheight)},
scale only axis,
xmode=log,
xmin=0.0001,
xmax=0.1,
xminorticks=true,
xlabel style={font=\color{white!15!black}},
xlabel={$\delta t$},
ymode=log,
ymin=1e-16,
ymax=0.01,
yminorticks=true,
ylabel style={font=\color{white!15!black}},
ylabel={$\mathcal{E}_{E,\delta t}$},
axis background/.style={fill=white},
legend style={legend cell align=left, align=left, draw=white!15!black}
]
\addplot [color=mycolor1, mark=*, mark options={solid, fill=mycolor1, mycolor1}]
  table[row sep=crcr]{%
0.03143	0.000141658448933848\\
0.0101	1.38751713595483e-05\\
0.00325714598997343	1.41645710900191e-06\\
0.00103	1.40782159355194e-07\\
0.000325714598997343	1.40507954044484e-08\\
0.000103	1.40408371419412e-09\\
};

\addplot [color=mycolor2, mark=square*, mark options={solid, fill=mycolor2, mycolor2}]
  table[row sep=crcr]{%
0.03143	2.83465266125492e-15\\
0.0101	5.78742136374352e-15\\
0.00325714598997343	1.31102811922791e-14\\
0.00103	2.30315750702752e-14\\
0.000325714598997343	3.4370196643375e-14\\
0.000103	1.84843153769356e-13\\
};

\addplot [color=mycolor3, mark=diamond*, mark options={solid, fill=mycolor3, mycolor3}]
  table[row sep=crcr]{%
0.03143	1.65354897940638e-15\\
0.0101	3.8976511657436e-15\\
0.00325714598997343	6.37797463485317e-15\\
0.00103	1.27559492697063e-14\\
0.000325714598997343	2.04331409598074e-14\\
0.000103	9.93310494057687e-14\\
};

\addplot [color=black]
  table[row sep=crcr]{%
0.03143	1e-06\\
0.0101	1e-07\\
};

\end{axis}
\end{tikzpicture}%
\end{minipage}\\    
  \end{tabular}
  \caption{$\mathcal{E}_{E,\delta t}$ for relaxation schemes and Crank-Nicolson scheme when $\sigma=2$ and $\sigma=3$ }
  \label{fig:un}
\end{figure}
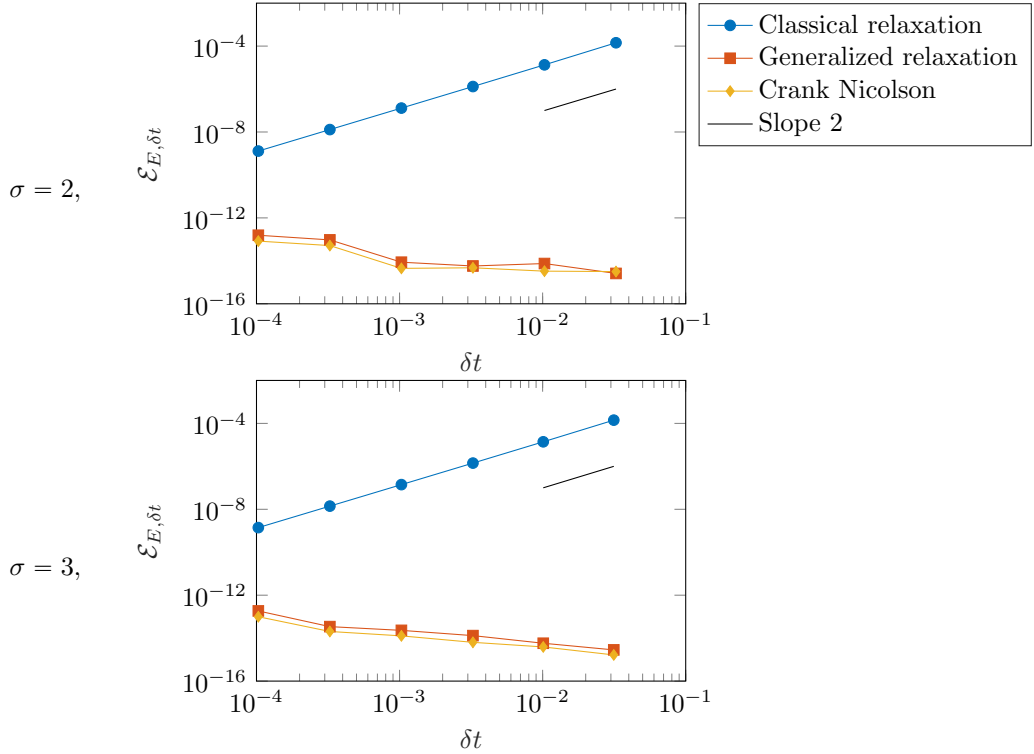

\subsection{A nonlocal Schr\"odinger equation with cubic-quintic nonlinearities}
In \cite{Krolikowski14}, Chen \textsl{et al.} investigate the interactions of dark solitons under competing nonlocal cubic and local quintic nonlinearities. They consider a 1D optical beam with an amplitude $\varphi(t,x)$ with competing nonlocal cubic and local quintic nonlinearities. Their model is given by the following nonlocal nonlinear Schr\"odinger equation
\begin{equation}
  \label{eq:cqnnls}
  \partial_t \varphi(t,x)=i \left ( \frac{1}{2} \partial_x^2 + \alpha_1 U \ast |\varphi(t,x)|^2 + \alpha_2 |\varphi(t,x)|^4 \right )\varphi(t,x),\quad \varphi(0,x)=\varphi_{\mathrm{in}},
\end{equation}
where $U(x)$ is the nonlocal kernel and $\alpha_1$ and $\alpha_2$ are real parameters. When $\alpha_j>0$, $j=1,2$, we are considering focusing nonlinearities, whereas when $\alpha_j<0$, the nonlinearities are defocusing.  This model is generalized in two dimensions in \cite{Shenetal2014} for the study of vortex solitons where $\partial_x^2$ is replaced by the two-dimensional Laplace operator. Typically, the kernel is regular and is given by
\begin{equation}
  \label{eq:kernel1}
  U_1(x)=\frac{\mathbbm{1}_{\{|x|\leq \mu\}} }{(2\mu)^d},\quad x\in \R^d,
\end{equation}
or
\begin{equation}
  \label{eq:kernel2}
  U_2(x)=\frac{1}{(\mu \sqrt{\pi})^d} \exp\left ( -|x|^2/\mu^2\right ),\quad x\in \R^d,
\end{equation}
the parameter $\mu$ allowing to control the width of the kernel, or in other words, the nonlocality strength. At the limit $\mu \to 0$, the kernel $U_j$, $j=1,2$, tends to a Dirac distribution and we recover a local nonlinear model.
The equation \eqref{eq:cqnnls} is associated to the energy
\begin{equation}
  \label{eq:3}
  E(\varphi)(t)=\frac12 \int_{\R}  \left ( \frac{1}{2}\|\nabla \varphi\|^2-\frac{\alpha_1}{2} (U\ast |\varphi|^2)|\varphi|^2 - \frac{\alpha_2}{3} |\varphi|^6\, \right )dx.
\end{equation}
We reproduce the  numerical experiments presented in \cite{Krolikowski14,Shenetal2014} and show the ability of the generalized relaxation method to preserve the energy \eqref{eq:NRJrlx} in the form \eqref{eq:nrj_grlx_nloc} below. The generalized relaxation method for \eqref{eq:cqnnls} consists in approximating the system of equations
\begin{equation}
  \label{eq:grl_nloc_cont}
  \left \{
  \begin{array}{l}
    \Upsilon=|\varphi|^2,\\[2mm]
    \gamma^2=|u|^4,\\[2mm]
\displaystyle    i\partial_t \varphi=\left (-\frac{1}{2} \Delta - \alpha_1 U \ast \Upsilon -\alpha_2 \gamma^2 \right ) \varphi,
  \end{array}
  \right .
\end{equation}
and the numerical scheme reads
\begin{equation}
  \label{eq:grl_nloc}
  \left \{
    \begin{array}{l}
      \Upsilon_{n+1/2}=2|\varphi_n|^2-\Upsilon_{n-1/2},\\[2mm]
      \displaystyle \gamma_{n+1/2}^2=\left ( \frac{3}{2} |\varphi_n|^2 -\gamma_{n-1/2}\right)(\gamma_{n-1/2}+\gamma_{n+1/2}),\\[3mm]
      \displaystyle i\frac{\varphi_{n+1}-\varphi_n}{\delta t}=-\frac{1}{2} \Delta \frac{\varphi_{n+1}-\varphi_n}{2} - \left (\alpha_1 U\ast \Upsilon_{n+1/2} +\alpha_2 \gamma_{n+1/2}^2 \right ) \frac{\varphi_{n+1}+\varphi_n}{2},
    \end{array}
  \right .
\end{equation}
with $\varphi_0(x)=\varphi_{\mathrm{in}}(x))$ and $\Upsilon_{-1/2}(x)=\gamma_{-1/2}(x)$ is some second order approximation of $\varphi(-\delta t/2,x)$. In our numerical experiments, this approximation is obtained by applying the Crank-Nicolson scheme starting from $\varphi_0$ on reverse time step $-\delta t/2$. The energy associated to \eqref{eq:grl_nloc_cont} is
\begin{equation}
  \label{eq:nrj_grlx_nloc}
  E_{\mathrm{rlx}}(\varphi,\gamma,\Upsilon) = \frac12 \int_{\R} \left(\frac{1}{2}\|\nabla \varphi\|^2 - \alpha_1 U\ast\Upsilon (|\varphi|^2 - \frac{\Upsilon}{2}) - \alpha_2 \gamma^2(|\varphi|^2-\frac{2}{3} \varphi)\right) \, dx
\end{equation}
and we have the conservation property \eqref{eq:proptheoremNRJpreservation}.

We are first interested in the one-dimensional case with kernel $U_1$ and we choose a defocusing nonlocal nonlinearity by considering $\alpha_1=-1$. Like in the previous subsection, the space variable is discretized using Fourier spectral approximation and we consider periodic boundary conditions. In order to avoid any interaction between the nonlocal kernel and the boundaries, we take a very large domain, typically $x\in [-256\pi,256\pi]$ discretized with $J=2^{14}+1$ nodes. The time step is $\delta t=5\cdot 10^{-3}$ and the final time is $T=30$. The initial datum is made of two solitons at a relative distance of $2x_0$, where $x_0=1$. Following \cite{Krolikowski14}, we choose
\[
\varphi_{\text{in}}(x)=\mathrm{tanh}(D(x-x_0))\mathrm{tanh}(D(x+x_0)),
\]
where $D$ is the positive root of the equation
\[
  \frac{\mathrm{coth(D\mu)}}{D\mu}\left [ \frac{1}{D^2}-\mu^2\mathrm{csch}^2(D\mu)\right]-\frac{11}{15}\frac{2\alpha_2}{3D^2}=\frac{1}{3}.
\]
We present results in Figure \ref{fig:nonlocA} and \ref{fig:nonlocB} for a defocusing cubic nonlinearity $\alpha_2=-0.5$ and a focusing one where $\alpha_2=0.1$. Two values for $\mu$ are proposed $\mu=0.5$ and $\mu=2.5$. The behavior of the defocusing-defocusing case is surprising for ``strong'' nonlocality $\mu=2.5$ since eventually the two solitons breathe.
\begin{figure}[h!]
  \centering
  \begin{tabular}{cc}
    \includegraphics[width=.48\textwidth]{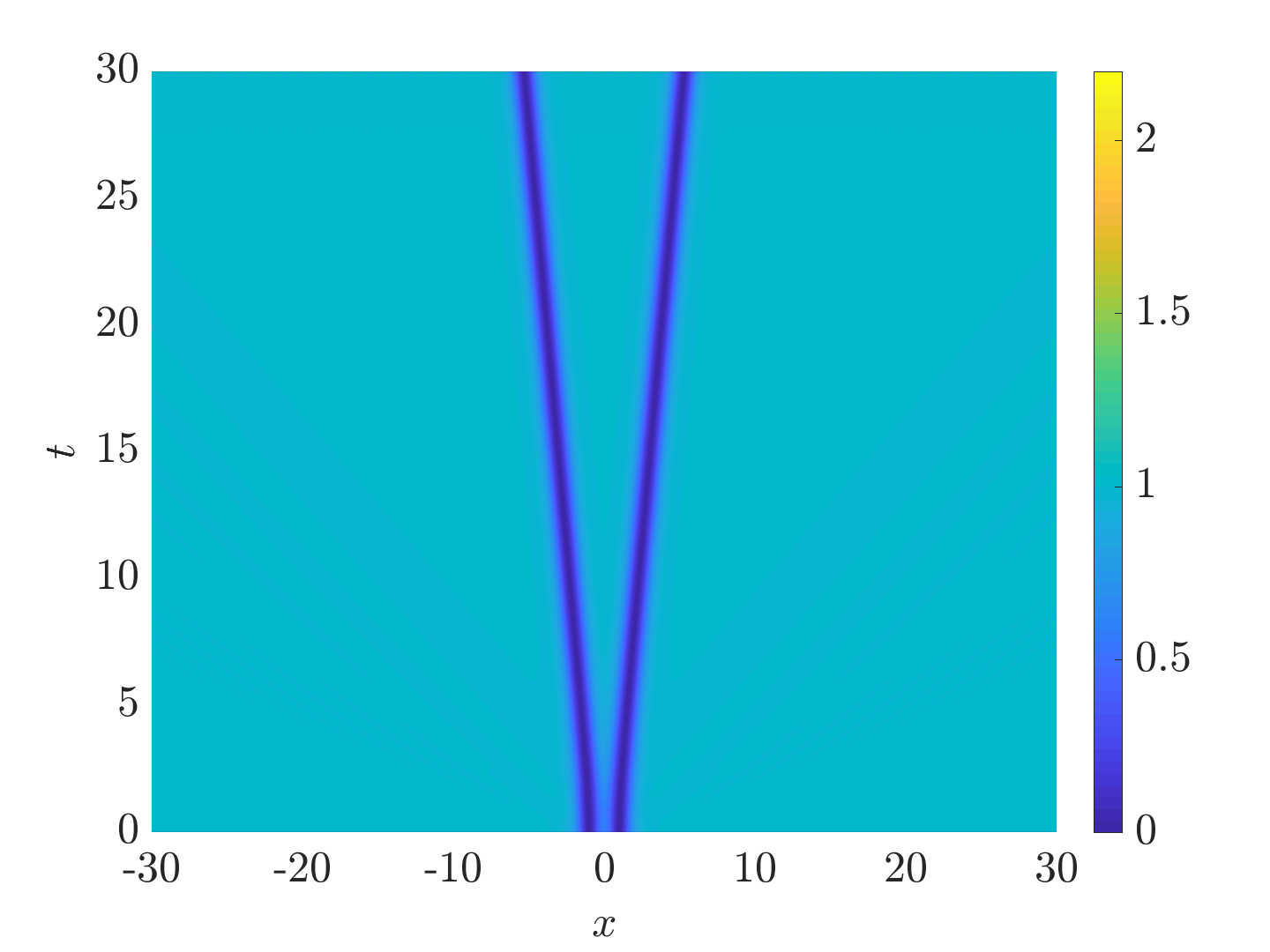}&    \includegraphics[width=.48\textwidth]{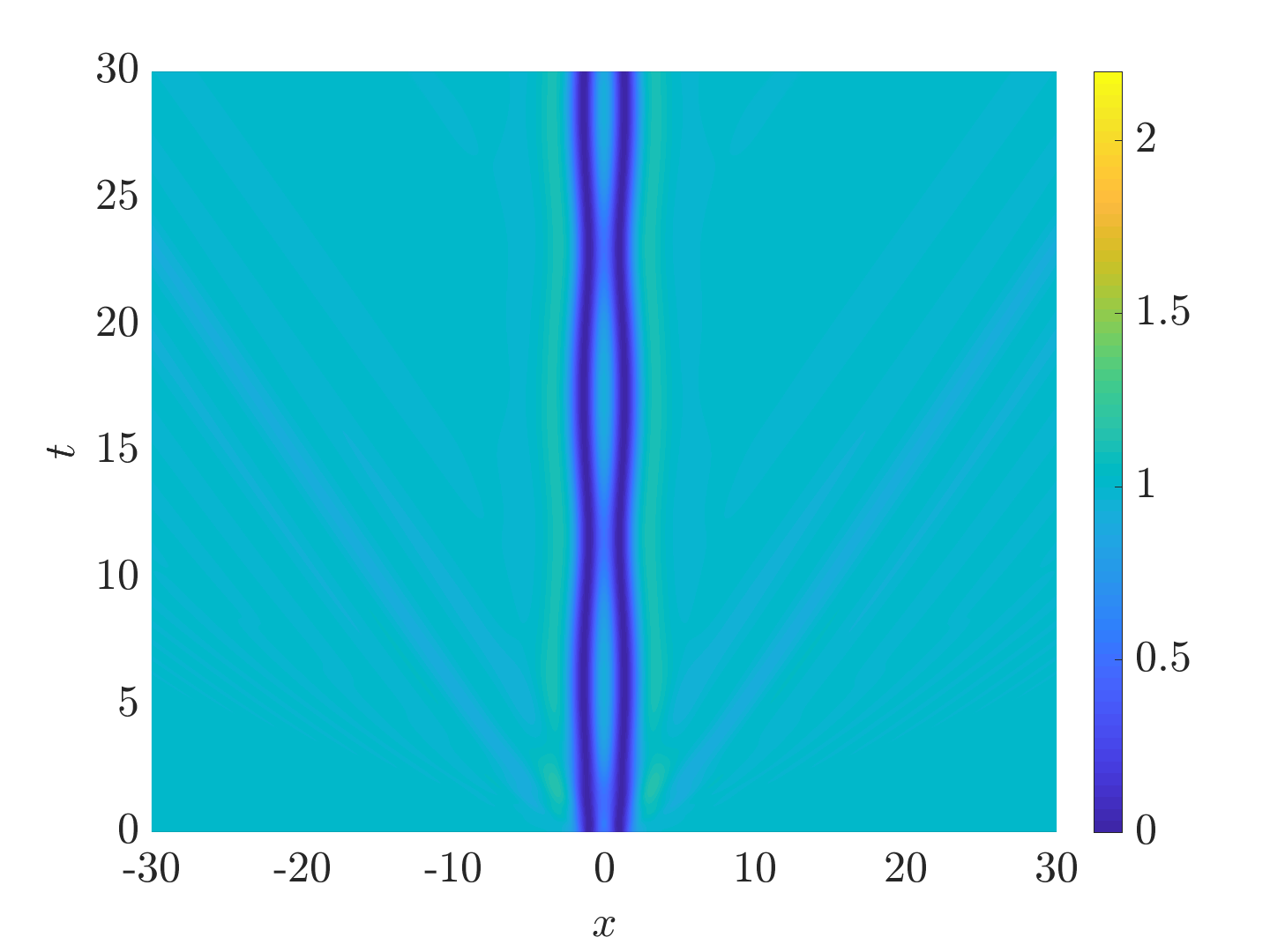} \\
    $\mu=0.5$ & $\mu=2.5$\\
  \end{tabular}
  \caption{Solutions $|\varphi_n|^2$ of the generalized relaxation method \eqref{eq:grl_nloc} applied to equation \eqref{eq:cqnnls} at time $T=30$ for $\alpha_2=-0.5$, $\mu=0.5$ and $\mu=2.5$.}
  \label{fig:nonlocA}
\end{figure}

\begin{figure}[h!]
  \centering
  \begin{tabular}{cc}
    \includegraphics[width=.48\textwidth]{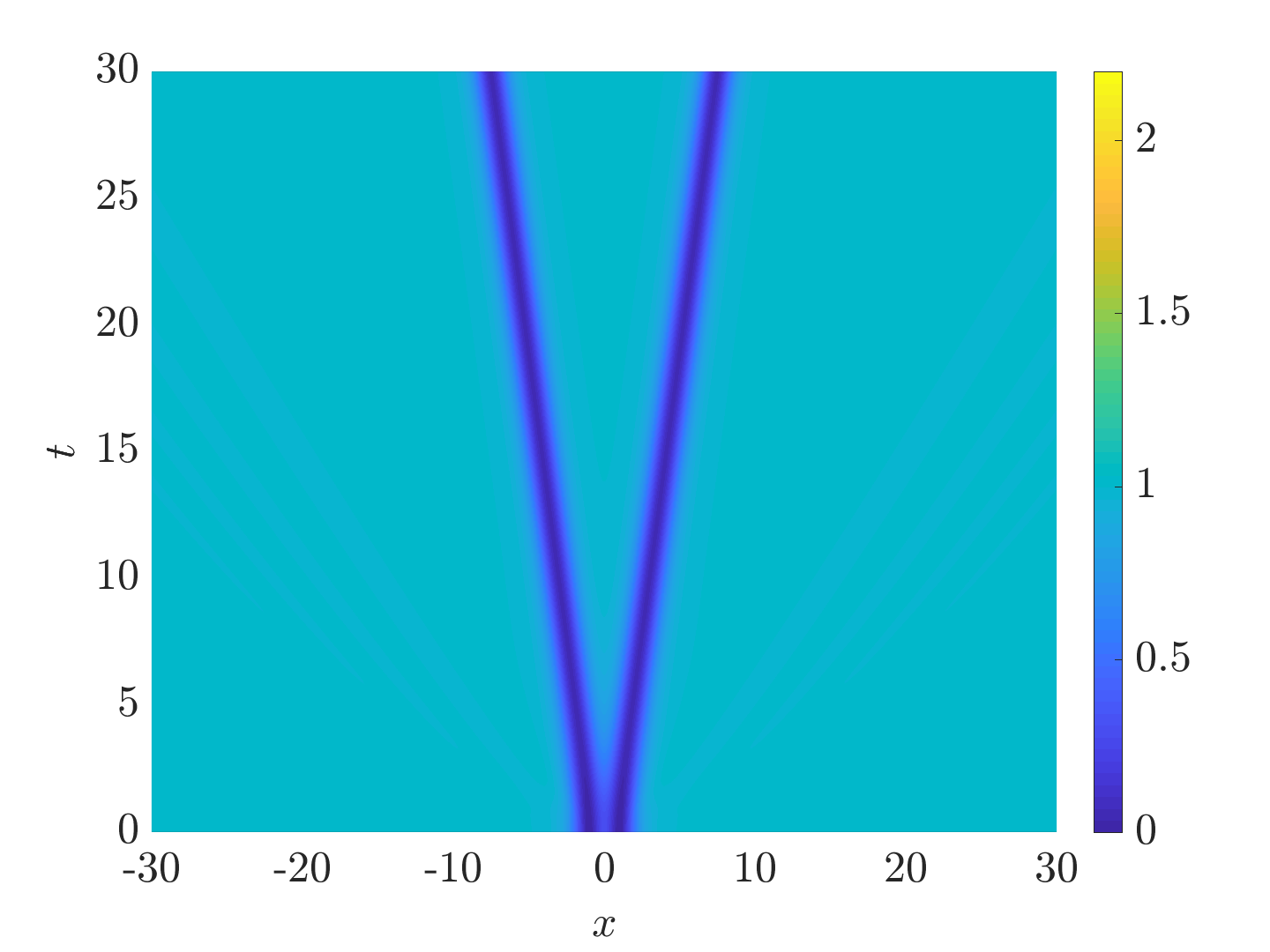}&    \includegraphics[width=.48\textwidth]{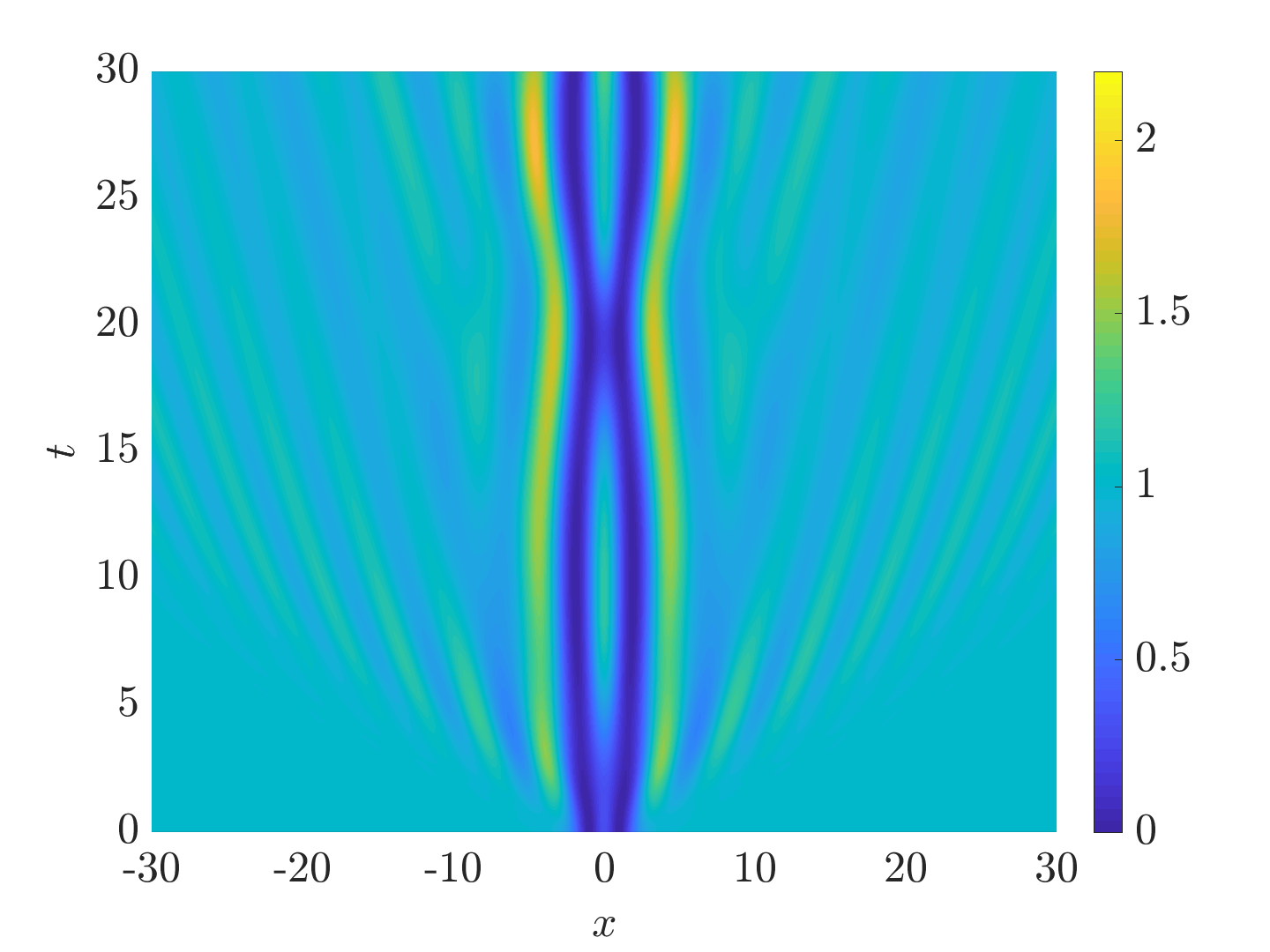} \\
    $\mu=0.5$ & $\mu=2.5$\\
  \end{tabular}
  \caption{Solutions $|\varphi_n|^2$ of the generalized relaxation method \eqref{eq:grl_nloc} applied to equation \eqref{eq:cqnnls} at time $T=30$ for $\alpha_2=0.1$, $\mu=0.5$ and $\mu=2.5$.}
  \label{fig:nonlocB}
\end{figure}
The evolution of the relative energy error
  \begin{equation}
\frac{\left | E_{\mathrm{rlx}}(\varphi_{n},\gamma_{n-1/2},\Upsilon_{n-1/2})-E_{\mathrm{rlx}}(\varphi_0,\gamma_{-1/2},\Upsilon_{-1/2})\right |}{E_{\mathrm{rlx}}(\varphi_0,\gamma_{-1/2},\Upsilon_{-1/2})}\label{eq:specialGuillaume}
\end{equation}
is presented on Figure \ref{fig:evol_ener_nloc}. The energy is clearly very well preserved.
\begin{figure}[h!]
  \centering
  \begin{tabular}{cc}
    \includegraphics[width=.48\textwidth]{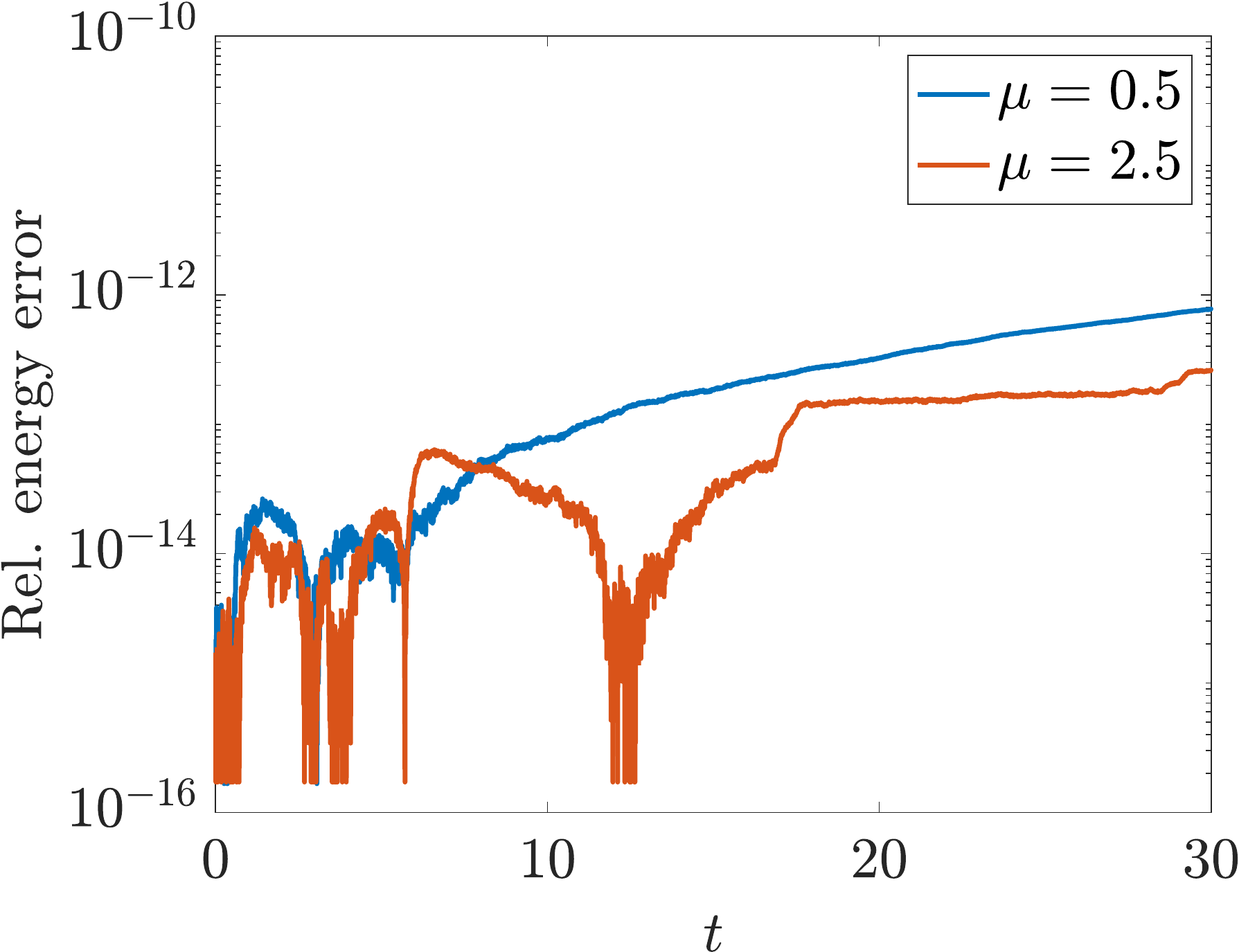}& \includegraphics[width=.48\textwidth]{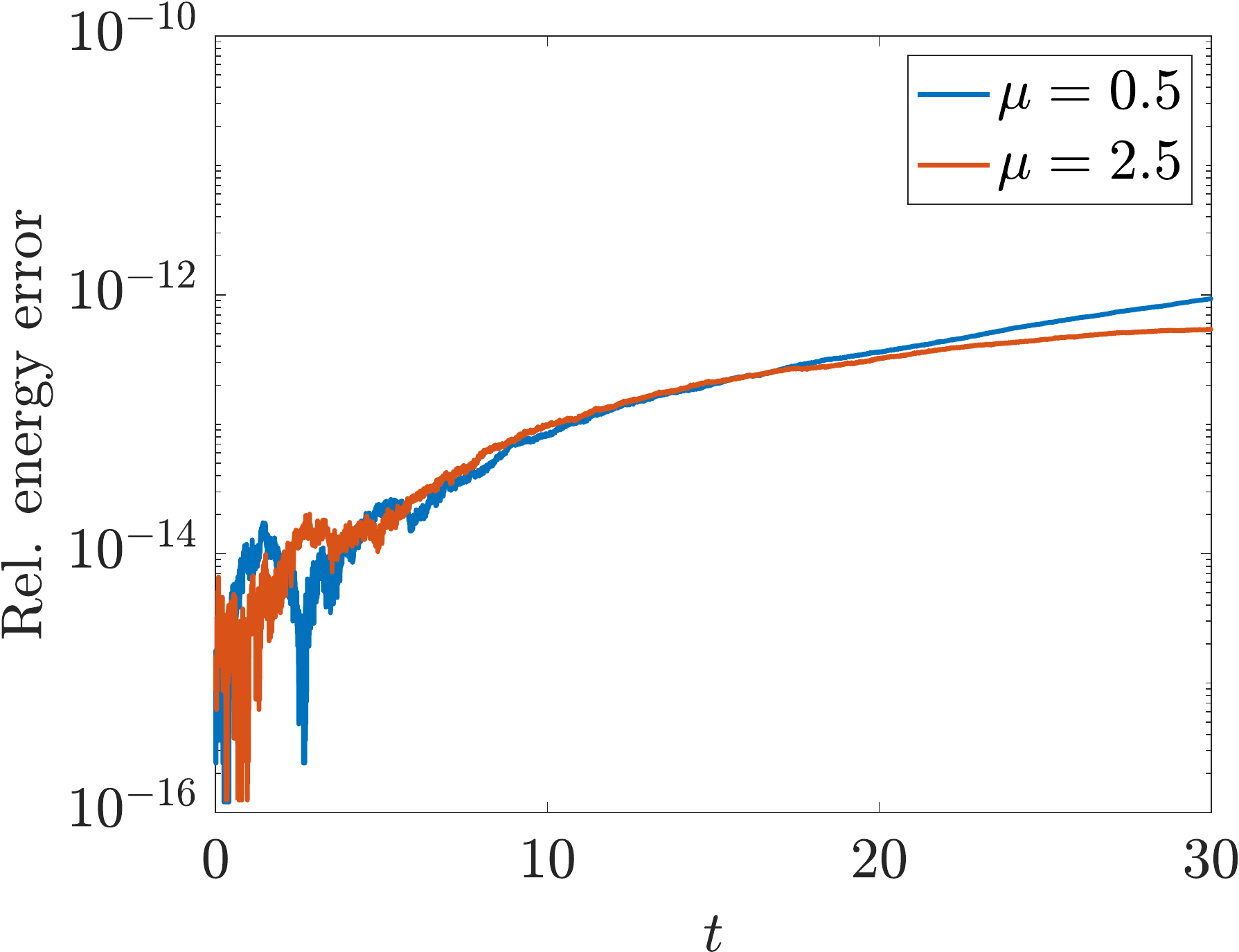} \\
    Nonlinear parameter $\alpha_2=-0.5$ & Nonlinear parameter $\alpha=0.1$\\
  \end{tabular}
  \caption{Evolution of the relative energy error \eqref{eq:specialGuillaume} for the generalized relaxation method \eqref{eq:grl_nloc} applied to equation \eqref{eq:cqnnls}}
  \label{fig:evol_ener_nloc}
\end{figure}

Concerning the two dimensionsal case, we consider the kernel $U_2$, $\alpha_1=1$ and $\alpha_2=-0.02$. The initial datum is chosen as a vortex beam with angular momentum
\[
\varphi_{\mathrm{in}}(x_1,x_2)=A r^m e^{-r^2/2}e^{im\phi},
\]
where $r=\sqrt{x_1^2+x_2^2}$, $m=1$ is the topological charge, $\phi$ is such that $x_1+ix_2=r\exp(i\phi)$ and $A=5.8$ is the amplitude. The computational domain is $[-8,8]^2$ and we use $256$ Fourier modes. The final time of simulation is $T=10$ and the time step is $5\cdot 10^{-3}$. The width of the Gaussian kernel $U_2$ is determined by $\mu=0.4$. Once again, the energy is very well preserved, since the relative error energy is bounded by $10^{-11}$ at the end of the simulation (see Figure \ref{fig:evol_ener_nonloc_2D}). We present in Figure \ref{fig:nonloc2D} the evolution of the solution with respect to time. The 3D representation of the time evolution is presented in figure \ref{fig:figA}. A 2D projection of it is displayed in Fig. \ref{fig:figB} and the final solution in Fig. \ref{fig:figC}. It is interesting to note that the same kind of breathing behaviour observed in the one dimensional setting is also present in 2D experiment.
\begin{figure}[h!]
  \centering
    \includegraphics[width=.48\textwidth]{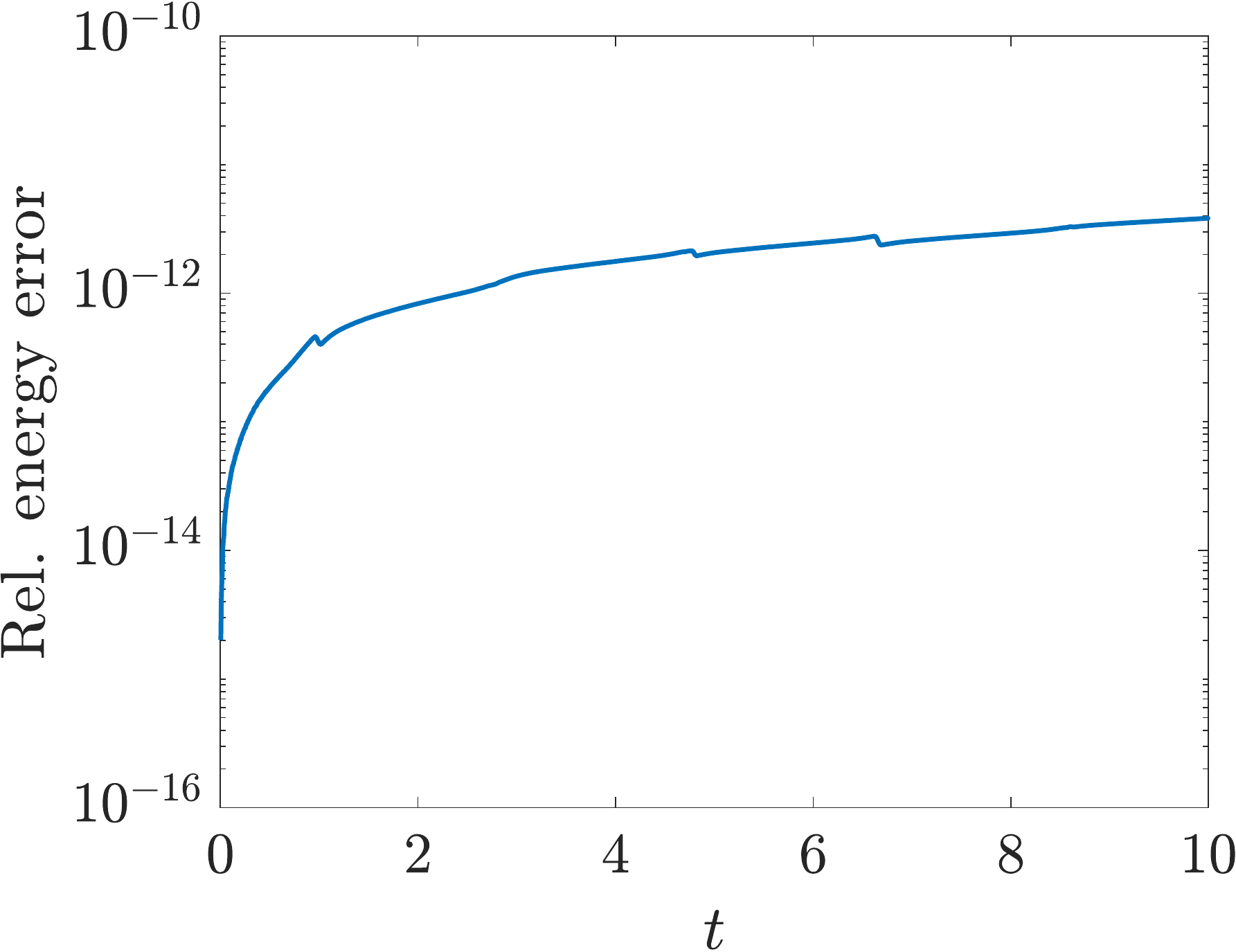}
  \caption{Evolution of the relative energy error \eqref{eq:specialGuillaume} for the generalized relaxation method \eqref{eq:grl_nloc} applied to equation \eqref{eq:cqnnls} for the two-dimensional case}
  \label{fig:evol_ener_nonloc_2D}
\end{figure}

\begin{figure}
\centering
\sbox{\measurebox}{%
  \begin{minipage}[b]{.6\textwidth}
  \subfloat
    [Evolution of $|\varphi(t,x_1,x_2)|^2$]
    {\label{fig:figA}\includegraphics[width=\textwidth]{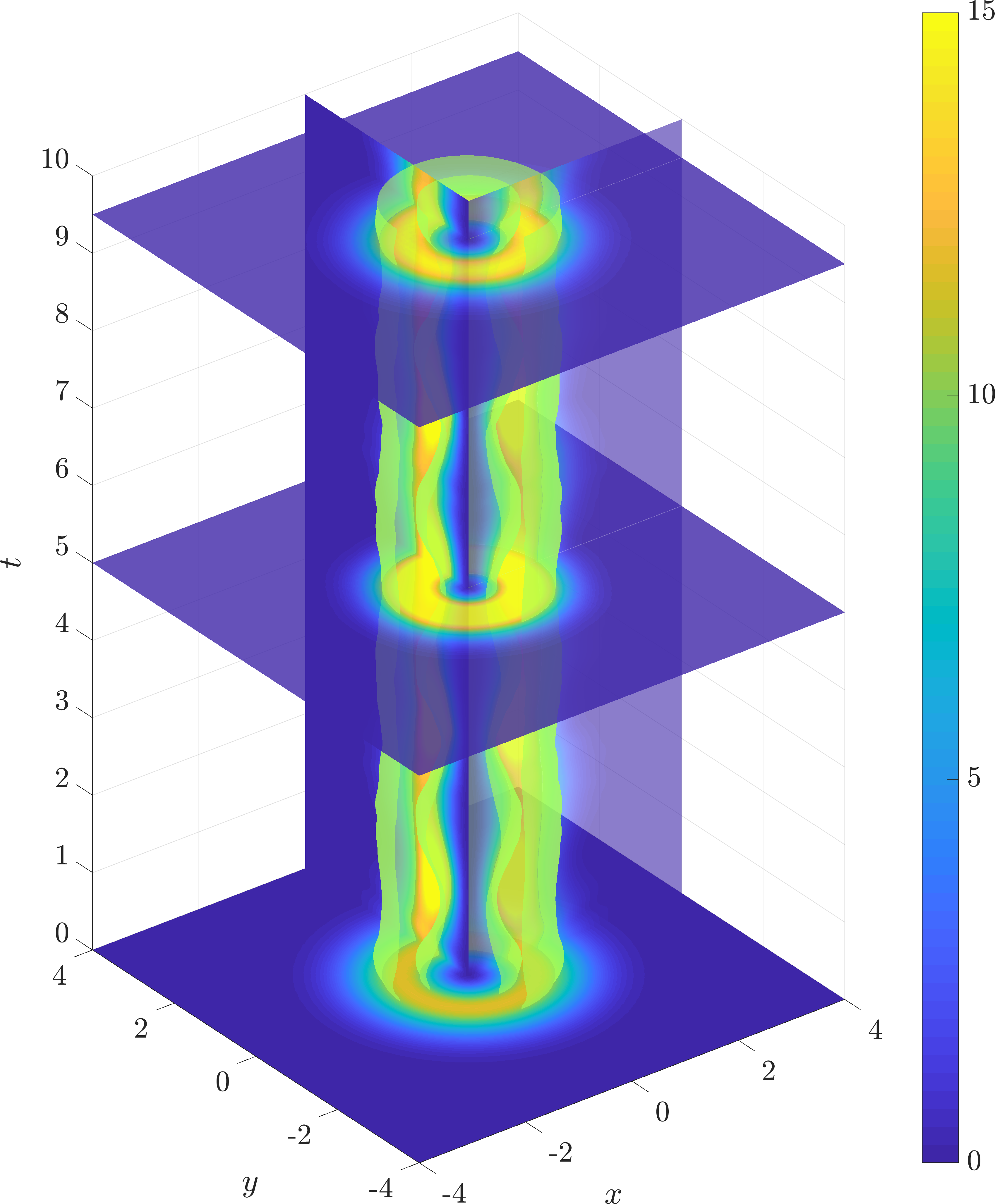}}
  \end{minipage}}
\usebox{\measurebox}\qquad
\begin{minipage}[b][\ht\measurebox][s]{.33\textwidth}
\centering
\subfloat
  [Evolution of $|\varphi(t,x_1,x_2)|^2$, slice at $x_2=0$]
  {\label{fig:figB}\includegraphics[width=\textwidth]{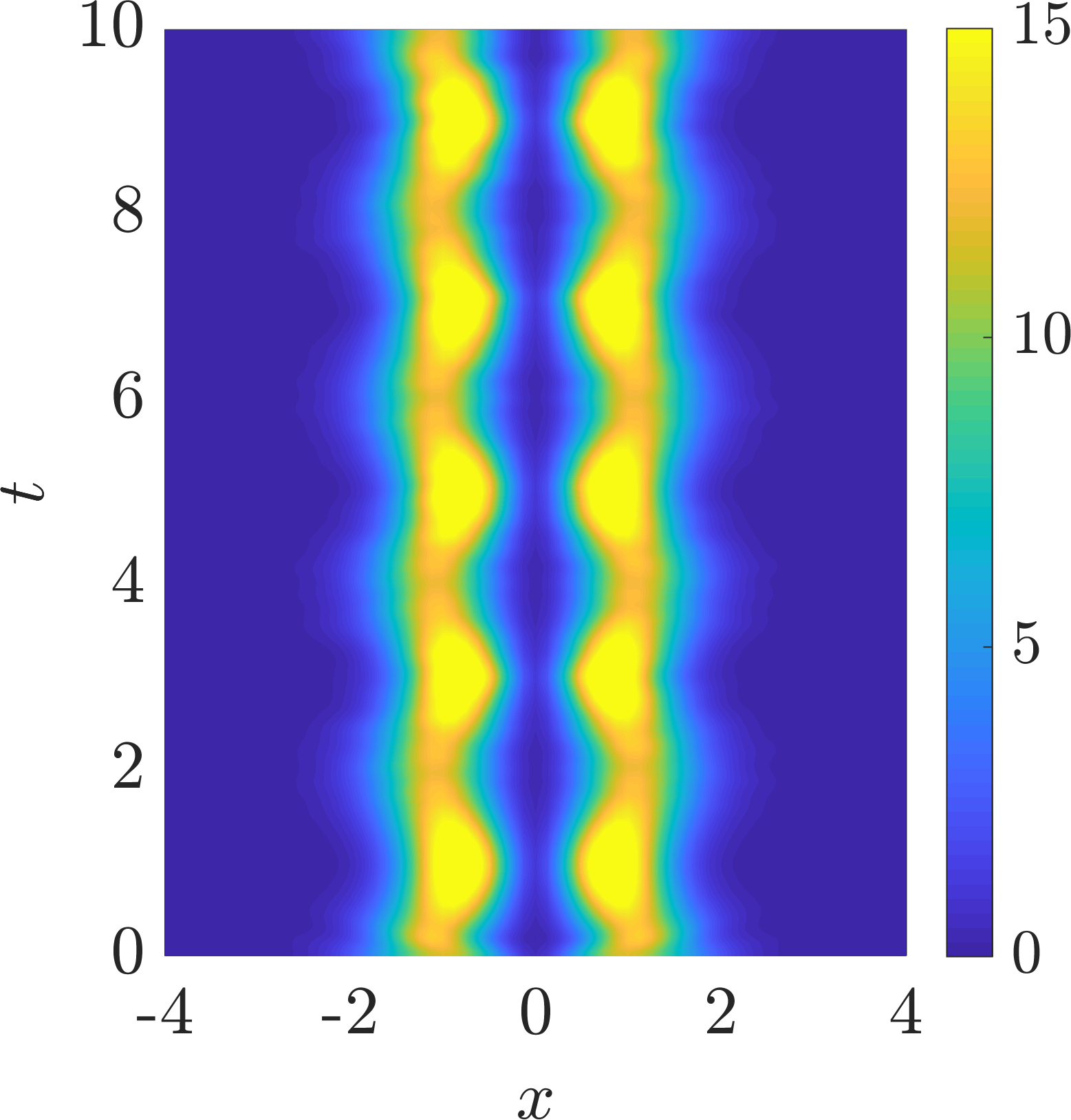}}

\vfill

\subfloat
  [Final state]
  {\label{fig:figC}\includegraphics[width=\textwidth]{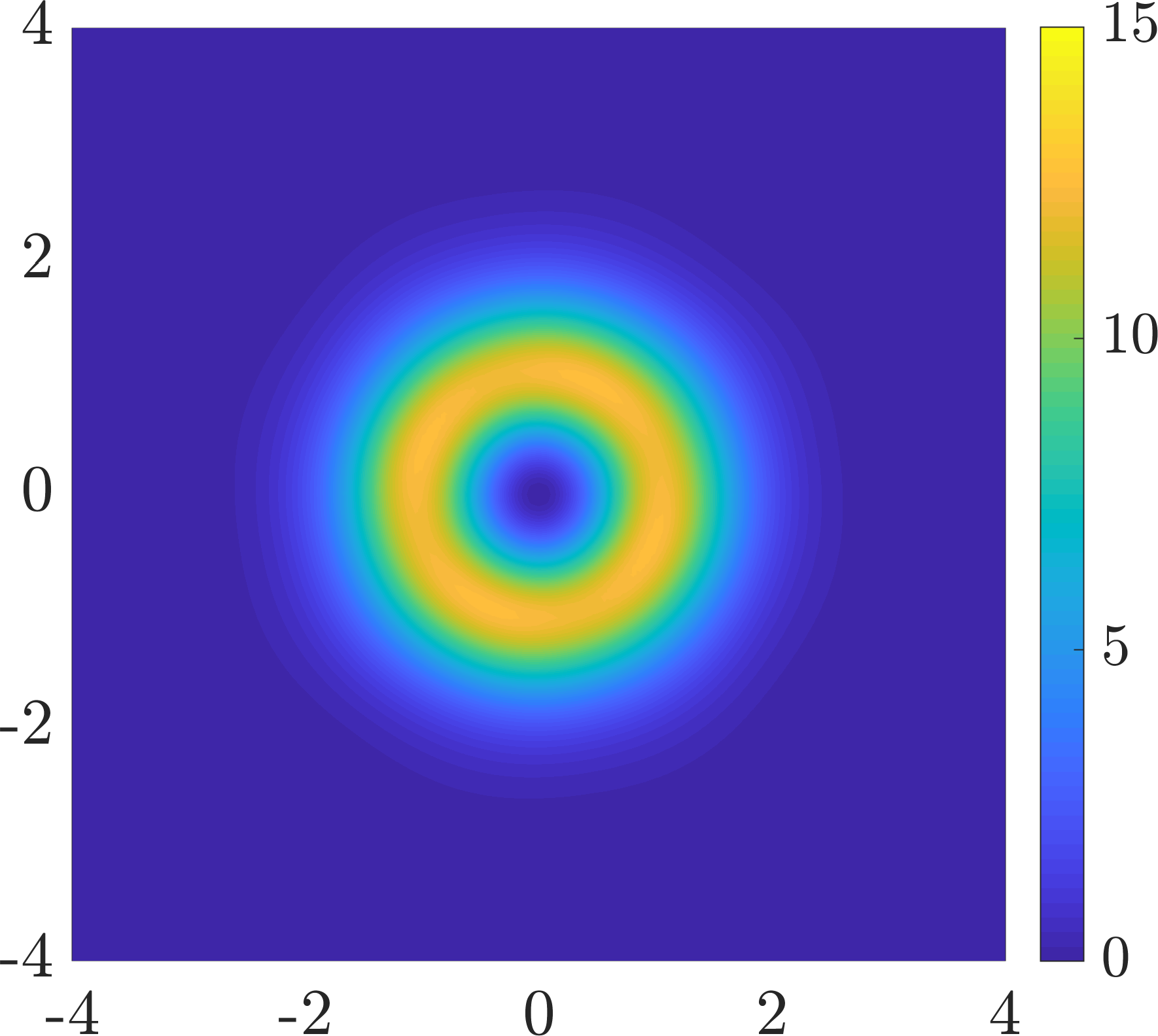}}
\end{minipage}
\caption{Evolution of $|\varphi(t,x_1,x_2)|^2$ for the generalized relaxation method \eqref{eq:grl_nloc} applied to equation \eqref{eq:cqnnls} for the two-dimensional case}
\label{fig:nonloc2D}
\end{figure}


\subsection{Two dimensional rotating dipolar Bose-Einstein condensate}
In this final subsection, we present results of simulations with the relaxation scheme for a rotating Bose-Einstein condensate subject to long range dipole-dipole interaction (DDI). The model is (see Eq. (8.15)-(8.17) in \cite{baocai})
\begin{equation}
  \label{eq:dipdip}
\left \{
  \begin{array}{l}
    \displaystyle
    i\partial_t \varphi (t,x) = \left(-\frac{1}{2}\Delta
    + V(x)
    + \beta |\varphi|^{2}(t,x)
    + \lambda \psi(t,x)
    -\Omega  L_{x_3} \right)
    \varphi(t,x), \quad x\in \R^2,\quad t>0,\\
    \displaystyle \psi(t,x)=-\frac{3}{2} (\partial_{\mathbf{n}_\perp \mathbf{n}_\perp}-n_3^2\Delta)\left ( \frac{1}{2\pi|x|} \ast |\varphi|^2 \right )(t,x), \quad x\in \R^2,\quad t>0,\\ 
    \varphi(0,x)=\varphi_0(x), \quad x\in \R^2,
  \end{array}
  \right .
\end{equation}
where $L_{x_3}=-i(x_1\partial_{x_2}-x_2 \partial_{x_1})$ is the $x_3$-component of the angular momentum and $\Omega$ represents the rotating frequency. The nonlinear parameters $\beta$ and $\lambda$ respectively describe the strength of the short-range two-body interactions
in the condensate and the strength of the dipolar interaction modeled with a Coulomb potential. The real-valued external trapping potential $V$ is chosen as $V(x)= (\gamma_{x_1}^2{x_1}^2+\gamma_{x_2}^2 {x_2}^2)/2$, where $\gamma_{x_1}>0$, $\gamma_{x_2}>0$ are dimensionless constants proportional to the trapping
frequencies in the both directions. The normal $\mathbf{n}=(n_1,n_2,n_3)^T$ represents the dipole axis and we define $\mathbf{n}_\perp=(n_1,n_2)^T$ and $\partial_{\mathbf{n}_\perp}=\mathbf{n}_\perp \cdot \nabla$. The energy is given by
\begin{equation}
  \mathcal{E}(\varphi)=\frac12\displaystyle  \int_{\R^2} \left ( \frac{1}{2} |\nabla \varphi|^2+V(x)|\psi|^2+\frac{\beta}{2} |\varphi|^4+\frac{\lambda}{2} \psi|\varphi|^2-\Omega \overline{\varphi}L_z\varphi \right ) dx.
  \label{eq:ener_dipdip}
\end{equation}
The initial datum is computed as a minizer of the energy on the intersection of the unit sphere of $L^2$ with the energy space
\[
\varphi_{\mathrm{in}}=\mathrm{arg}\min_{\|\phi\|^2=1}\mathcal{E}(\phi).
\]
We use a preconditioned nonlinear conjugate gradient method developed in \cite{AnLeTan2017}. Contrary to the previous subsection where the convolution kernel was regular, the computation of the nonlocal term $\psi$ with the Coulomb potential is known to be a costly task. We have chosen to apply the technique developed in \cite{ViGrFe2016}, which allows fast convolution using truncated Green's functions. The parameters of the simulation are $\Omega=0.97$, $\gamma_{x_1}=\gamma_{x_2}=1$, $\lambda=175$, $\beta=(250-\lambda)\sqrt{5/\pi}$. The computational domain is $(-16,16)^2$ with $2^8=256$ Fourier modes in each direction, and the computational time is set to $T=15$ with the time step $\delta t=10^{-3}$. Initially, the dipole axis is $\mathbf{n}=(1,0,0)^T$ and we change it to $\mathbf{n}=(\cos(\pi/3),\sin(\pi/3),0)^t$.
We plot on Figure \ref{fig:evol_dipdip} the evolution of the solution from time $t=0$ to time $t=20$. We add the direction of the dipole axis on each frame. The axis are removed to make the presentation clearer. As it can be seen on Figure \ref{fig:evol_ener_dipdip}, the energy is very well preserved as expected.
\begin{figure}[h!]
  \centering
  \begin{tabular}{ccc}
    \includegraphics[width=.3\textwidth]{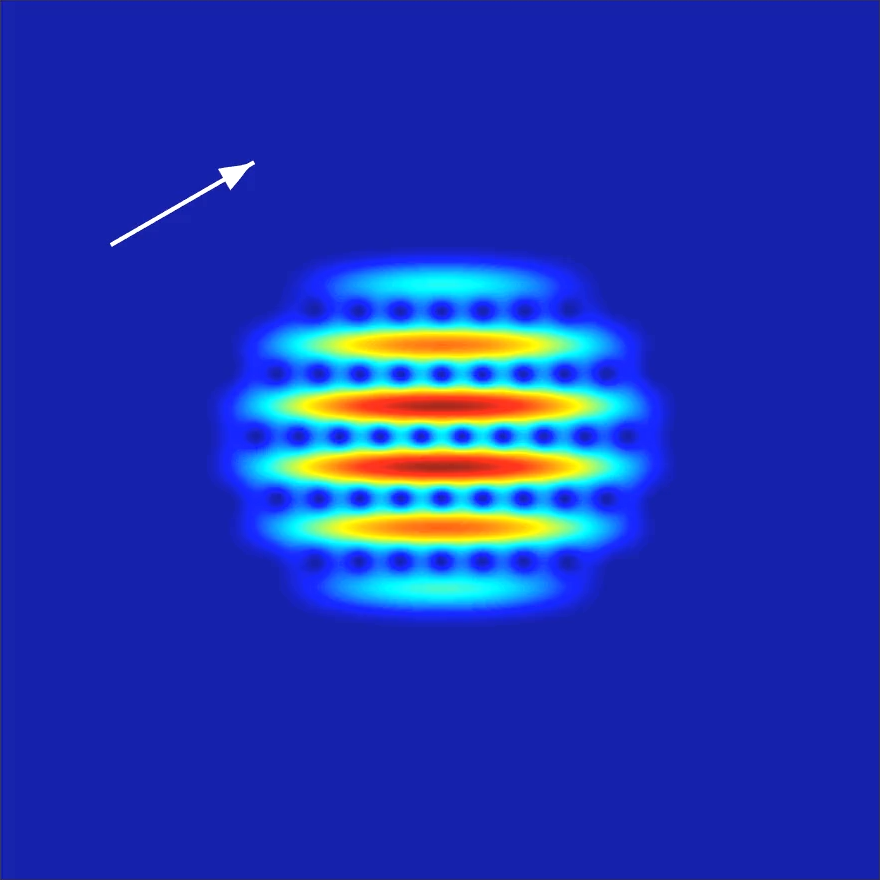}& \includegraphics[width=.3\textwidth]{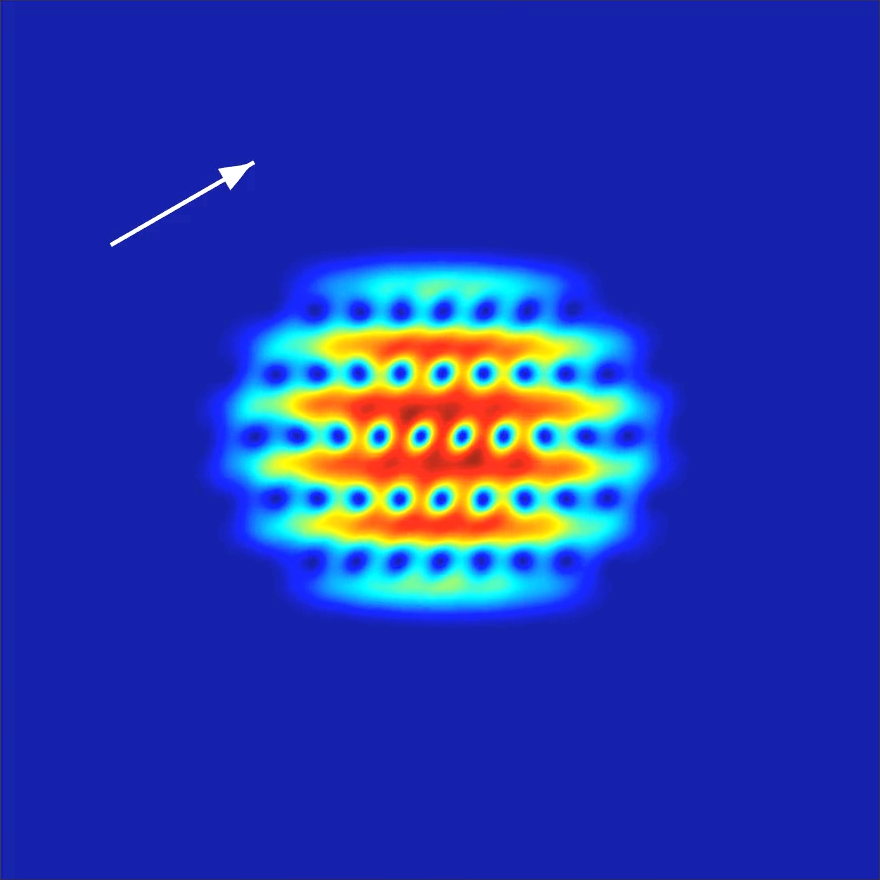} & \includegraphics[width=.3\textwidth]{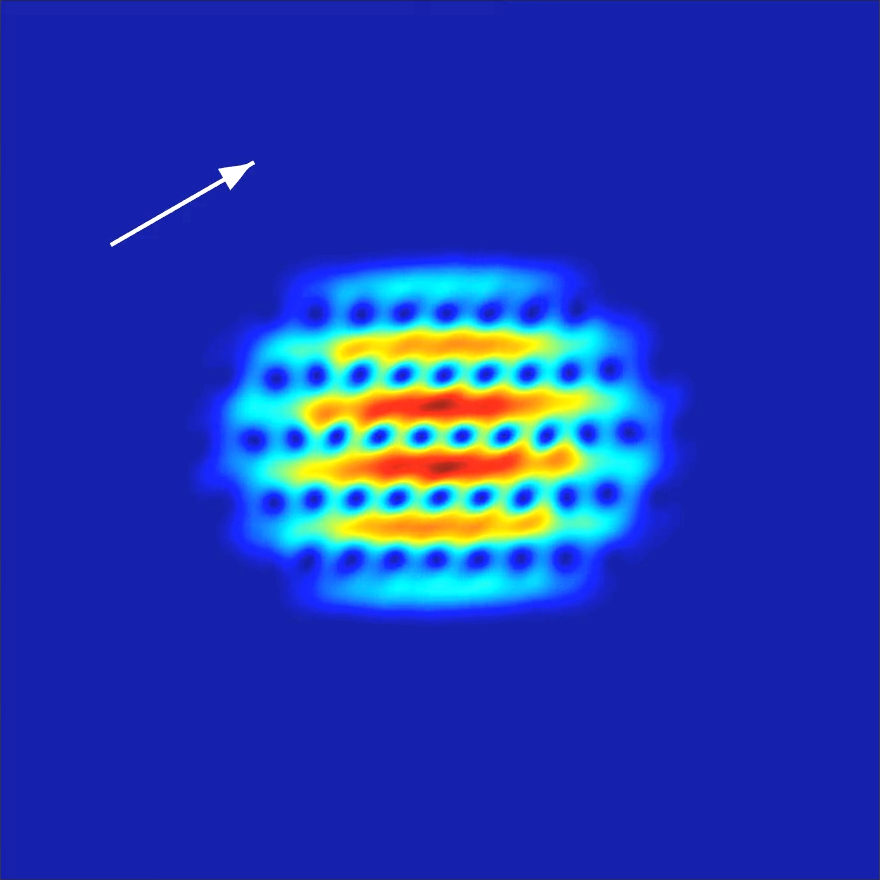} \\
    $t=0$ & $t=2.5$ & $t=5.0$\\[2mm]
    \includegraphics[width=.3\textwidth]{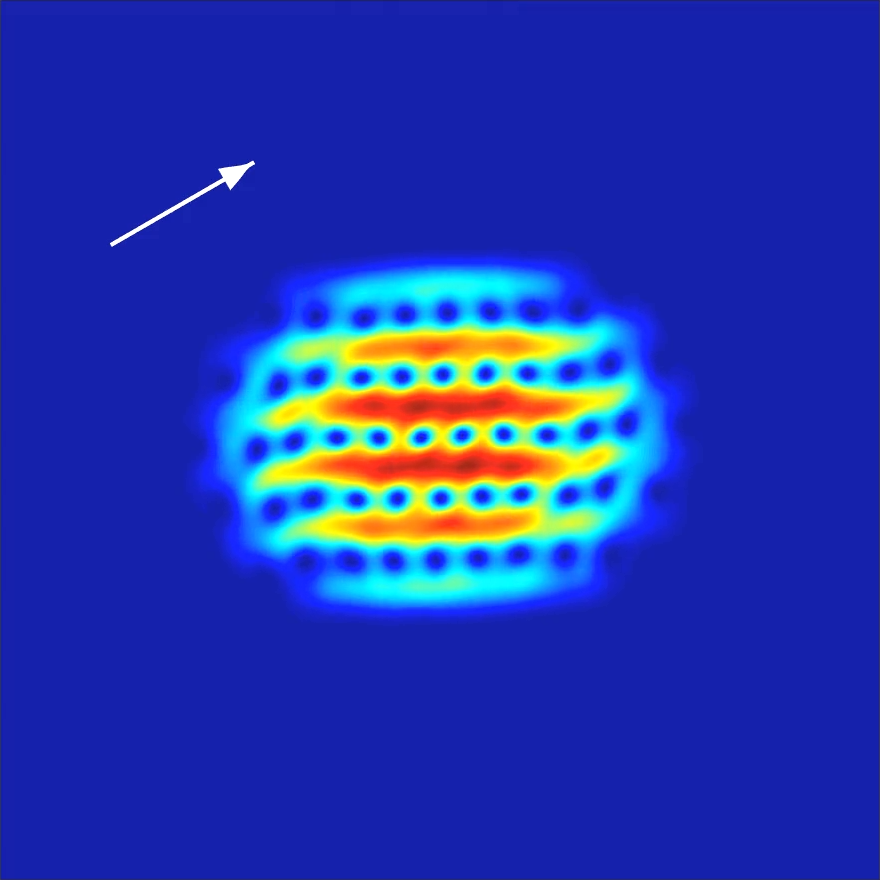}& \includegraphics[width=.3\textwidth]{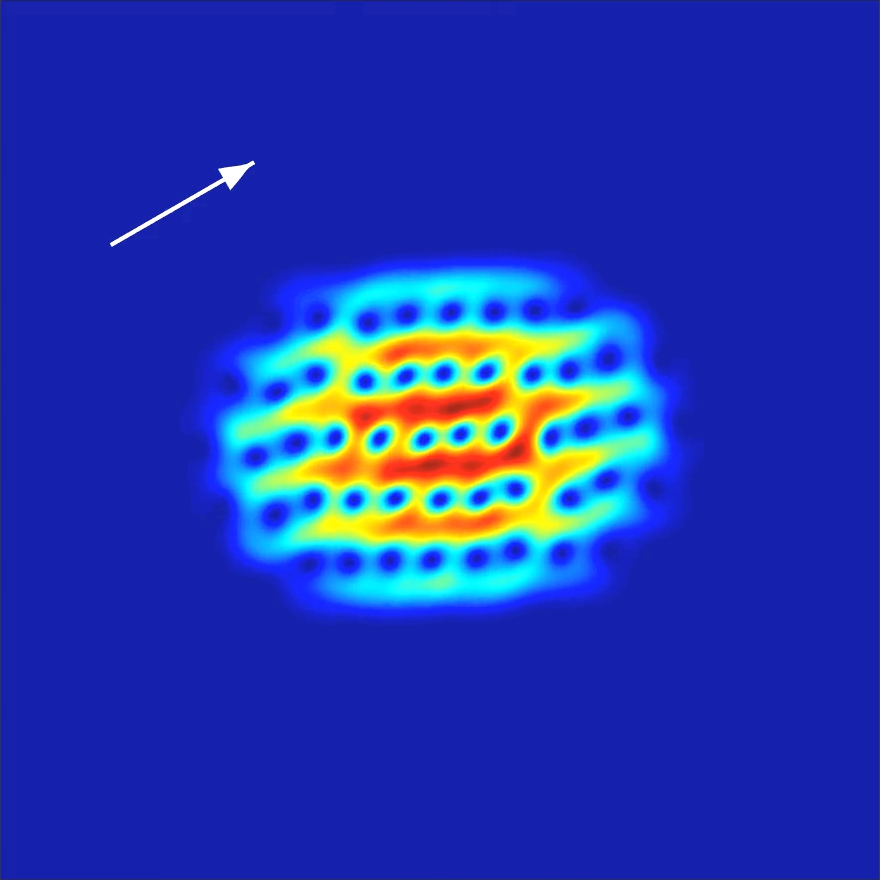} & \includegraphics[width=.3\textwidth]{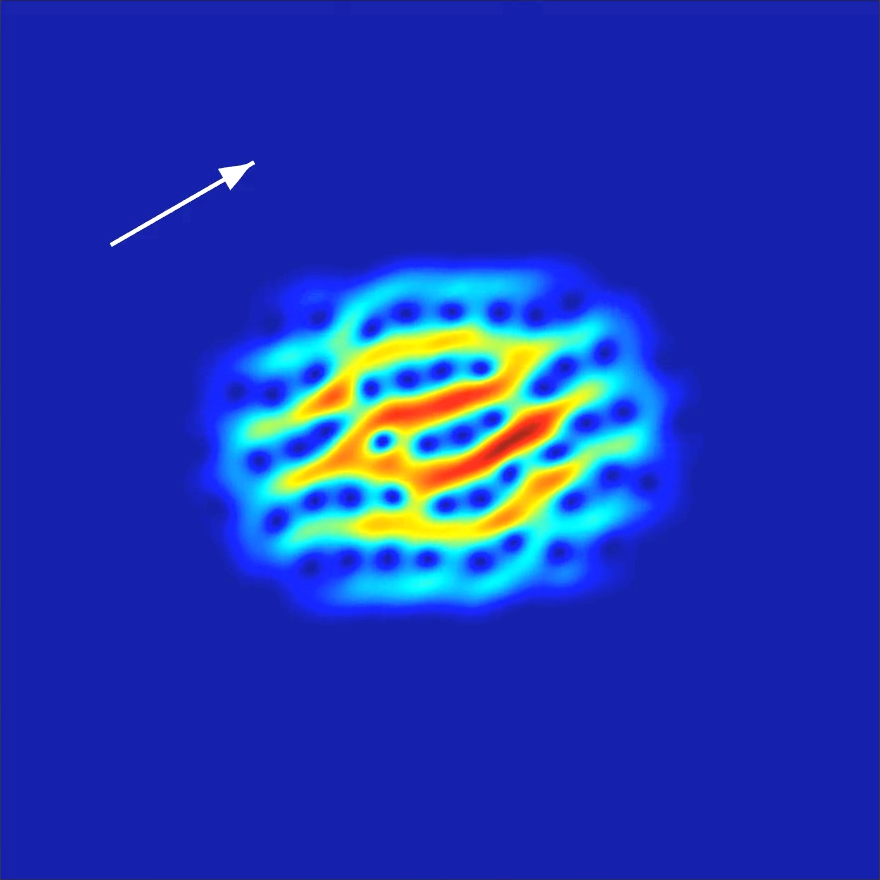} \\
    $t=7.5$ & $t=10$ & $t=12.5$\\[2mm]
    \includegraphics[width=.3\textwidth]{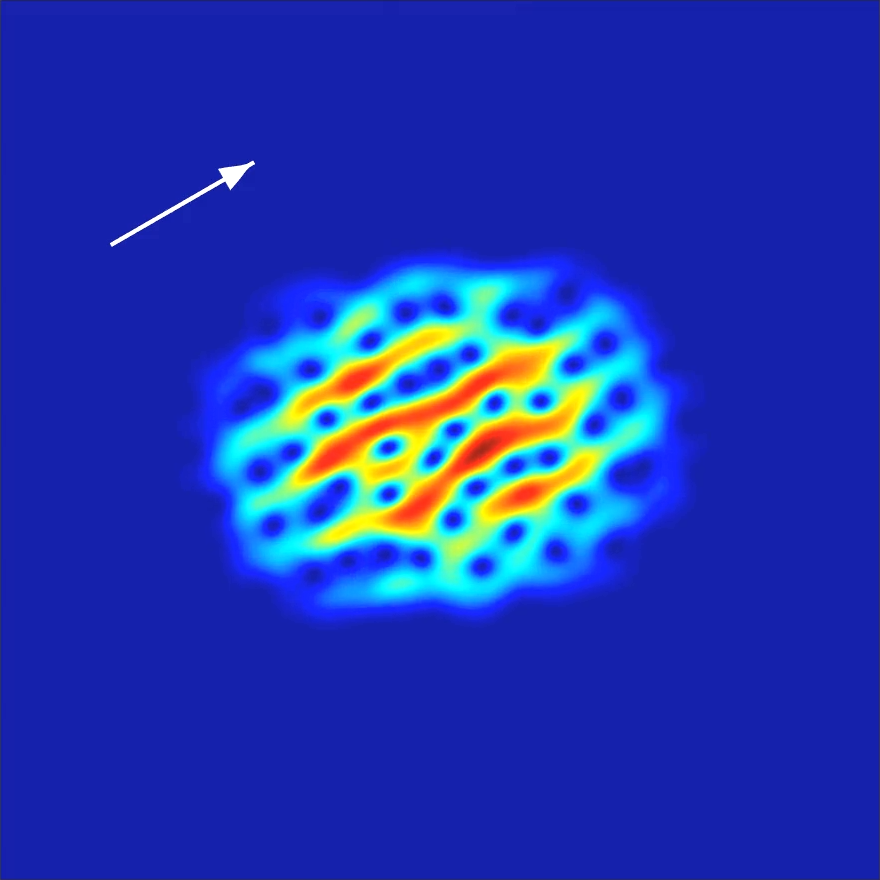}& \includegraphics[width=.3\textwidth]{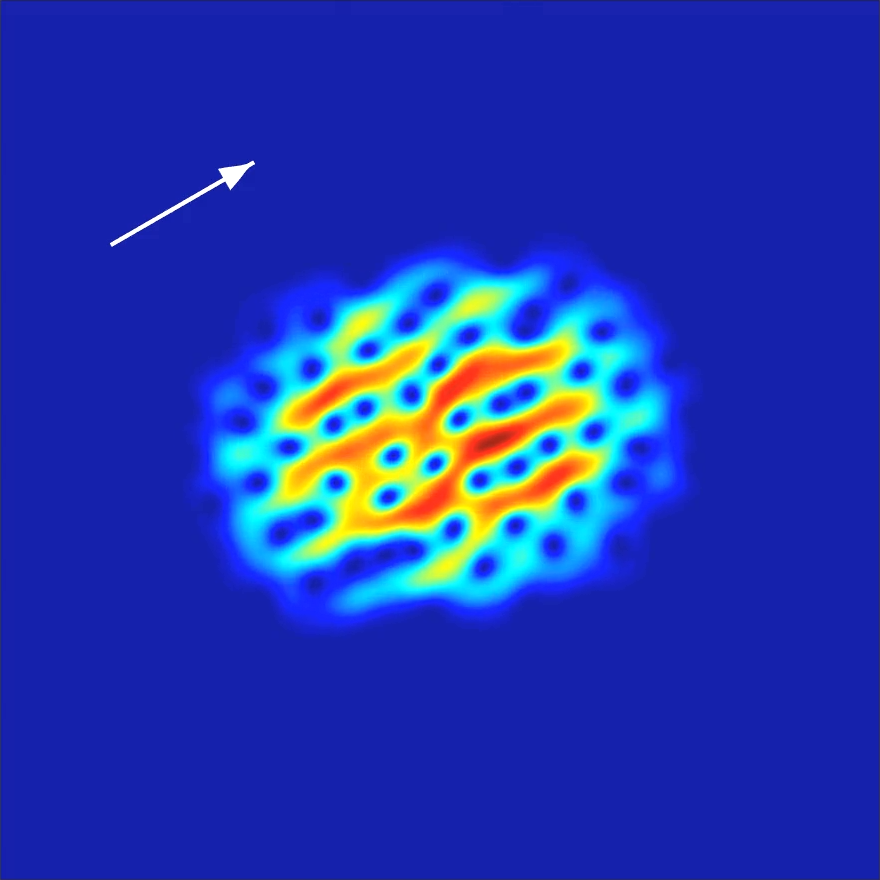} & \includegraphics[width=.3\textwidth]{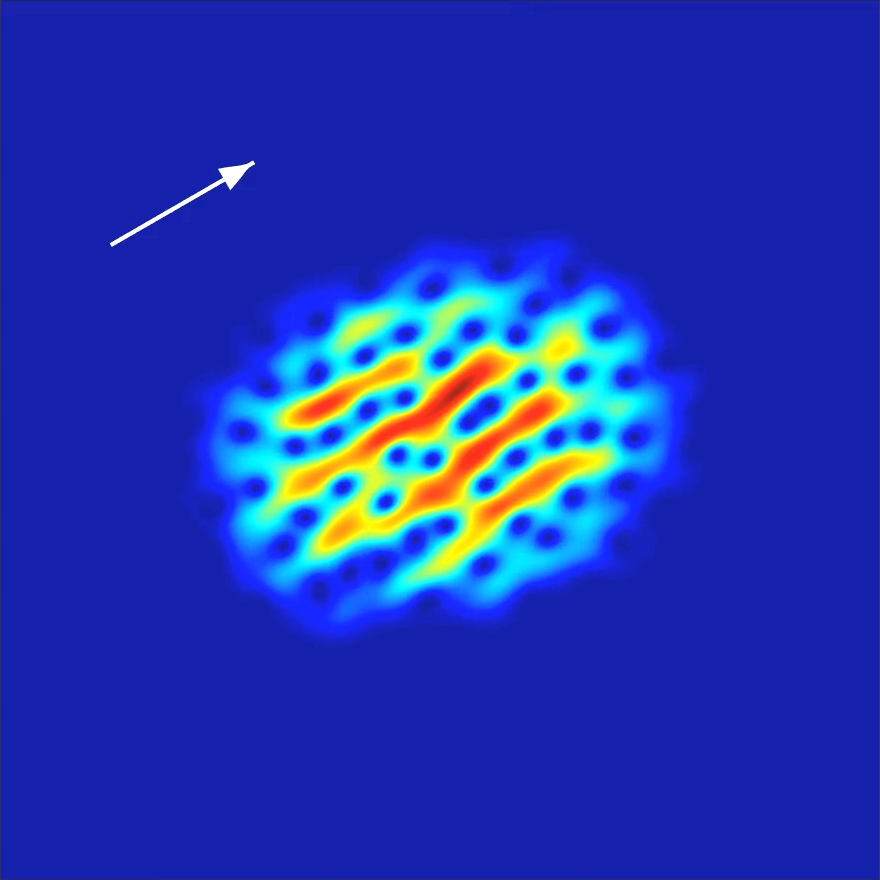} \\
    $t=15$ & $t=17.5$ & $t=20$\\
  \end{tabular}
  \caption{Evolution of the solution $|\varphi(t,x)|^2$ of Eq. \eqref{eq:dipdip} modeling rotating dipolar Bose-Einstein condensate}
  \label{fig:evol_dipdip}  
\end{figure}

\begin{figure}[h!]
  \centering
  \includegraphics[width=.48\textwidth]{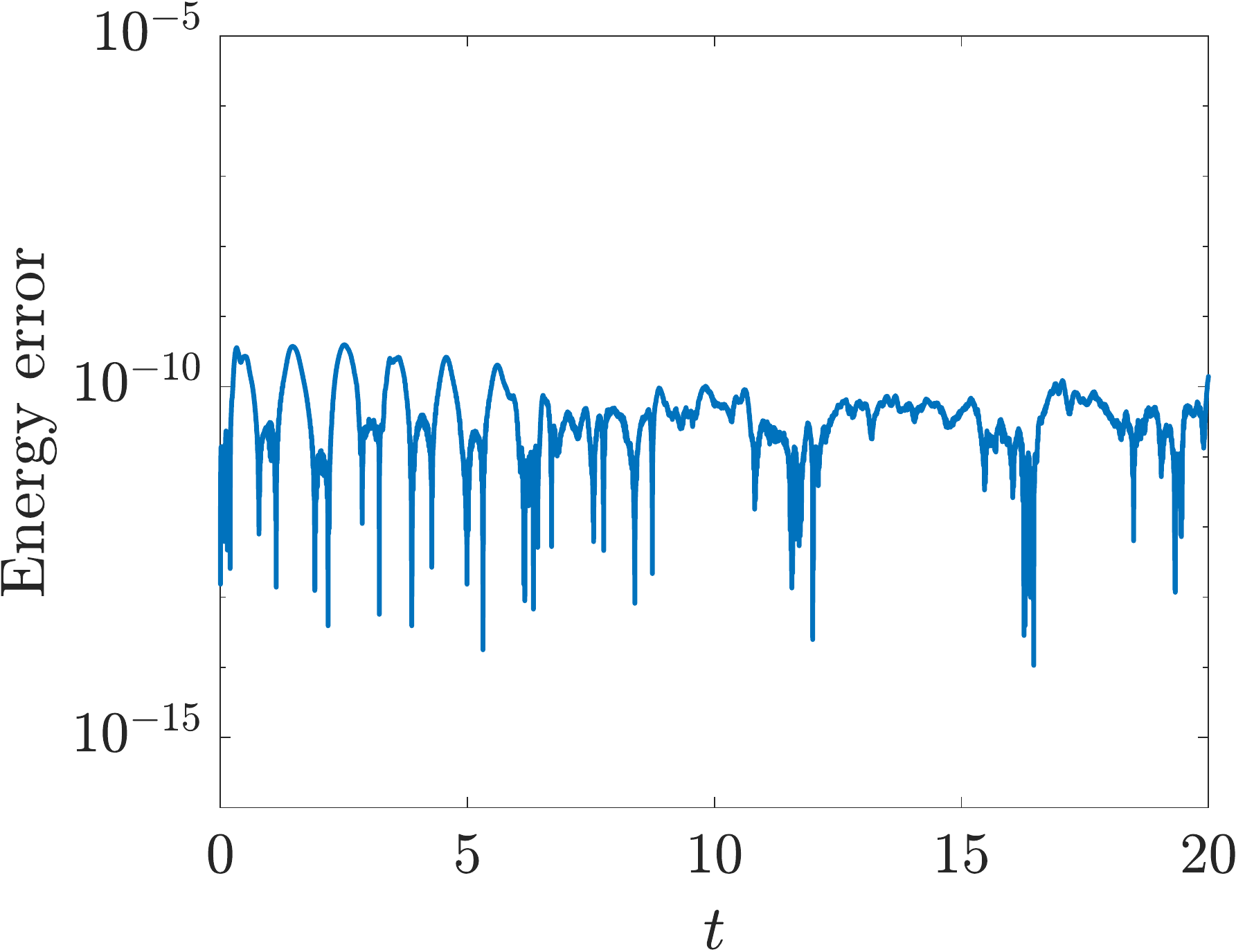}
  \caption{Evolution of the energy error with respect to time for the relaxation method \eqref{Frelaxation} applied to equation \eqref{eq:dipdip}}
  \label{fig:evol_ener_dipdip}
\end{figure}

\section{Conclusions and future works}
In this paper we have given for the first time a proof of the second order of
the relaxation method introduce in \cite{BessePhD} for the cubic nonlinear
Schr\"odinger equation.
We also have extended the previous method to deal with general power law
nonlinearites showing that this method is still an energy preserving method.
Note that, since Crank-Nicolson methods are one-step methods, one may
use composition techniques to achieve higher orders (4,6,8, {\it etc}) while
preserving both an energy and the $L^2$-norm,
and this is not possible so straightforwardly for relaxation methods.
Therefore, in future works we want to focus on energy preserving relaxation
methods which have higher order.

\appendix
\section{Notation}\label{sec:notation}

In this section, the symbol $\mathcal F$ will denote different operators
depending on whether $x$ belongs to $\R^d$ or $x$ belongs to $\T_\delta^d$.
When $x\in\R^d$, $\mathcal F$ denotes the Fourier transform
$\R^d$ defined for $\varphi\in L^1(\R^d)$ via the formula
\begin{equation*}
  {\mathcal F}(\varphi)(\xi) = \frac{1}{(2\pi)^{d/2}}
  \int_{\R^d} \varphi(x) {\rm e}^{-i x.\xi} \dd x,
\qquad \xi\in\R^d,
\end{equation*}
and its inverse ${\mathcal F}^{-1}$ is defined for $\phi \in L^1(\R^d)$
through the formula
\begin{equation*}
  {\mathcal F}^{-1}(\phi)(x) = \frac{1}{(2\pi)^{d/2}}
  \int_{\R^d} \phi(\xi) {\rm e}^{+i x.\xi} \dd \xi,
\qquad x\in\R^d.
\end{equation*}
When $x\in\T_\delta^d=(\R/(\delta\Z))^d$, the
$\mathcal F$ denotes the Fourier transform
defined for $\varphi\in L^1(\T_\delta^d)$ via the formula
\begin{equation*}
  {\mathcal F}(\varphi)(\xi) = \frac{1}{\delta^{d}}
  \int_{\T_\delta^d} \varphi(x) {\rm e}^{-i x.\xi} \dd x,
\qquad \xi\in\frac{2\pi}{\delta}\Z^d,
\end{equation*}
and its inverse ${\mathcal F}^{-1}$ is defined for
$\phi \in \ell^1(2\pi\delta^{-1}\Z^d)$
through the formula
\begin{equation*}
  {\mathcal F}^{-1}(\phi)(x) = \sum_{k\in\frac{2\pi}{d}\Z^d}
  \phi(k) {\rm e}^{+i x.k},
\qquad x\in\T_\delta^d.
\end{equation*}

\bibliographystyle{plain}
\bibliography{labib}

\end{document}